\documentclass[11pt]{article} 

\usepackage[utf8]{inputenc} 

\usepackage{geometry} 
\usepackage{graphicx} 
\graphicspath{ {./images/} }

\usepackage{booktabs} 
\usepackage{array,multirow} 
\usepackage{paralist} 
\usepackage{verbatim} 
\usepackage{subfig} 

\usepackage{fancyhdr} 
\usepackage{sectsty}
\allsectionsfont{\sffamily\mdseries\upshape} 

\usepackage[nottoc,notlof,notlot]{tocbibind} 
\usepackage[titles,subfigure]{tocloft} 

\usepackage{tikz}
\usetikzlibrary{positioning, fit, shapes}

\usepackage{amssymb}
\usepackage{amsmath}
\usepackage{amsthm}
\usepackage{mathtools}
\usepackage{planarforest}
\usepackage[ruled,vlined]{algorithm2e}
\usepackage{MnSymbol}
\usepackage{fge}
\usepackage[numbers,sort&compress]{natbib}
\usepackage{hyperref}

\newtheorem{assumption}{Assumption}
\newtheorem{theorem}{Theorem}[section]
\newtheorem{lemma}[theorem]{Lemma}
\newtheorem{proposition}[theorem]{Proposition}
\newtheorem{corollary}[theorem]{Corollary}
\newtheorem{corollary-definition}[theorem]{Corollary-Definition}

\newtheorem{definition}[theorem]{Definition}
\newtheorem{remark}[theorem]{Remark}
\theoremstyle{remark}
\newtheorem{example}[theorem]{Example}

\newcommand{\Aut}{\text{Aut}}

\newcommand{\graft}{\curvearrowright}

\newcommand{\gl}{\diamond}
\newcommand{\ckdual}{\circledast}

\newcommand{\E}{\mathbb{E}}
\newcommand{\N}{\mathbb{N}}

\newcommand{\Fg}{\mathcal{F}_g}

\newcommand{\EF}{\mathcal{EF}}
\newcommand{\AF}{\mathcal{AF}}
\newcommand{\AT}{\mathcal{AT}}

\newcommand{\A}{\mathcal{A}}

\newcommand\restrict[1]{\raisebox{-.5ex}{$|$}_{#1}}

\newcolumntype{C}[1]{>{\centering\let\newline\\\arraybackslash\hspace{0pt}}m{#1}}

\title{Exotic B-series and S-series: algebraic structures and order conditions for invariant measure sampling}

\author{Eugen Bronasco\textsuperscript{1}}

\begin{document}
\footnotetext[1]{
Université de Genève, Section de mathématiques, Genève, Switzerland. Eugen.Bronasco@unige.ch. \\
\indent \indent 7-9 rue du Conseil-Général, Case postale 64, 1211 Genève 4.}

\maketitle
\vspace{-0.6cm}
\textit{Communicated by Hans Munthe-Kaas.}

\begin{abstract}
B-series and generalizations are a powerful tool for the analysis of numerical integrators. An extension named exotic aromatic B-series was introduced to study the order conditions for sampling the invariant measure of ergodic SDEs. Introducing a new symmetry normalization coefficient, we analyze the algebraic structures related to exotic B-series and S-series. Precisely, we prove the relationship between the Grossman-Larson algebras over exotic and grafted forests and the corresponding duals to the Connes-Kreimer coalgebras and use it to study the natural composition laws on exotic S-series. Applying this algebraic framework to the derivation of order conditions for a class of stochastic Runge-Kutta methods, we present a multiplicative property that ensures some order conditions to be satisfied automatically. 
\end{abstract}


\section{Introduction}

The concept of B-series was introduced in the 1960s by John Butcher as a tool to study Runge-Kutta methods and generalizations of these. The idea of B-series originates from the fact that both Taylor expansions of the exact solution of an ODE and of numerical solutions obtained using Runge-Kutta methods can be written using formal sums indexed by rooted trees \cite{Butcher_63,Hairer_Wanner}.

For an arbitrary dimension $d\in\mathbb{N}$, let $f : \mathbb{R}^d \to \mathbb{R}^d$ be a vector field and let the covariant derivative at a point $p \in \mathbb{R}^d$ along the vectors $v_1, \dots, v_n \in \mathbb{R}^d$ be defined as
\[ f^{(n)} (p) (v_1, \dots, v_n) := \sum_{i_1, \dots, i_n = 1}^d v_1^{i_1} \cdots v_n^{i_n} \frac{\partial^n f}{\partial x_{i_1} \cdots \partial x_{i_n}} (p). \]
Then, B-series are defined using the correspondence between vector fields of the form $f(y_0)$, $f^\prime (y_0) f(y_0)$, $f^{\prime \prime} (y_0) (f(y_0), f(y_0))$, $\dots$ and the set of rooted non-planar trees of the form $\forest{b}, \forest{b[b]}, \forest{b[b,b]}, \dots$ \cite{Cayley}. This allows us to use combinatorics of rooted trees to study the properties and operations over vector fields.

Larger classes of trees and tree-like structures were introduced in the literature \cite{Butcher21bsa,ConnesKreimer,GeomIntBook} to fit different purposes. For example, bicolored trees, e.g. 
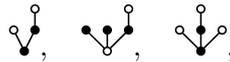
\[ \forest{b[w,b[w]]}, \quad \forest{w[b,b,b[w]]}, \quad \forest{b[w,b[w],w]}, \]
were introduced to study splitting methods and to define P-series that are used in the study of partitioned methods. The close relationship between vector fields and differential operators lead to the introduction of collections of trees called forests, e.g.
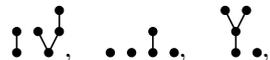
\[ \forest{b[b],b[b,b[b]]}, \quad \forest{b,b,b[b],b}, \quad \forest{b[b[b,b]],b}, \]
and their use to represent differential operators. The series that are based on forests are called S-series and they were originally introduced to study first integrals~\cite{Mur99}. 

Another generalization of trees, called aromatic trees, was introduced as a tool to study numerical integrators equivariant under affine maps. Aromatic trees are used to define aromatic B-series \cite{aromaticBseries} and were first introduced independently in \cite{ChartierMurua} and in \cite{Iserles2007E} as a way to express the divergence of the vector field of a problem. We note that rooted trees can be defined as directed trees with roots being the vertices with no outgoing edges. Aromatic trees\footnote{In graph theory terminology, aromatic trees are not trees because they may contain cycles. This name was proposed in \cite{ChartierMurua} in analogy with carbon chemistry.} are graphs with every vertex having at most one outgoing edge with exactly one vertex (the root) without an outgoing edge. A collection of aromatic trees is called an aromatic forest and the set of aromatic forests is used to define aromatic S-series \cite{Bogfjellmo}.
 An aromatic tree can have multiple connected components, examples are,
\[ \forest{(b,b[b,b]),b[b[b]]}, \quad \forest{b[b,b,b[b]]}, \quad \forest{(b),(b,b[b]),b[b]}, \]
where the directions of edges forming a cycle is counterclockwise. We extend aromatic trees and forests to the stochastic context and use S-series as the main tool.

Different generalizations of trees were introduced in the stochastic context for the study of the error on the trajectory (strong error) and the error on the law (weak error) of the numerical solution. Burrage and Burrage \cite{Burrage_Burrage_96} and Komori, Mitsui and Sugiura \cite{Komori_Mitsui_Sugiura} introduced stochastic trees and B-series for studying the order conditions for strong convergence of SDE, and \cite{Burrage_Burrage_00,Debrabant_Kvaerno_08,Rossler_06,Debrabant_Kvaerno_10,Rossler_04,Rossler2_06,Rossler_10} for study of high order weak and strong methods on a finite time interval. We consider ergodic integrators and recall that the weak order can be used to obtain convergence with respect to the sampling of the invariant measure~\cite{Talay_Tubaro}. However, there exist schemes which have a smaller error with respect to the sampling of the invariant measure than is predicted by their weak order, see \cite{Abdulle_Vilmart_Zygalakis,BouRabee10lra,Laurent2021-to,Laurent_2019}, and \cite{Laurent_2021}, where the order with respect to the invariant measure is considered. The order conditions for sampling the invariant measure, and their algebraic structures, are the main focus of this paper.

First studied in \cite{Laurent_2019}, grafted trees and exotic trees are the trees that correspond to vector fields appearing in the study of SDEs with additive noise. Due to the nice properties of the overdamped Langevin equation, one of which is ergodicity of the solution, we consider it as an example and study the related algebraic structures and the order conditions with respect to the invariant measure. Some examples of the grafted trees and exotic trees, with more examples in Table \ref{table:all_grafted_exotic_trees} of the Appendix, are
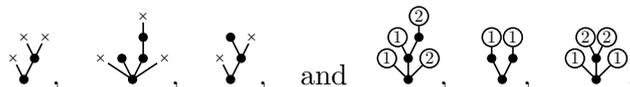
\[ \forest{b[x,b[x,x]]}, \quad \forest{b[x,b,b[b[x]],x]}, \quad \forest{b[x,b[b,x]]}, \quad \text{and} \quad \forest{b[1,b[1,b[2]],2]}, \quad \forest{b[b[1],b[1]]}, \quad \forest{b[1,b[2,2],1]}. \ \]

In this paper, we use B-series over grafted trees following \cite{Laurent_2019} and introduce a new normalization of the series using the symmetry coefficients $\sigma(\tau)$ of trees $\tau$ analogously to the deterministic case~\cite{ButcherSanzSerna}. Due to the fact that we work in a stochastic context, we study the expectation of the functional $\phi$ applied to the one step of a method $X_1 := \Psi_h(X_0, f)$, i.e. $\mathbb{E} [\phi(\Psi_h(X_0, f))]$. This leads us to the notion of S-series and we introduce S-series over grafted forests and exotic S-series.

In Sections \ref{sec:Asigma} and \ref{sec:exotic_S-series}, we use the combinatorial algebra framework of decorated aromatic forests to describe the relationship between S-series over grafted forests and exotic S-series, and present composition laws for the new kinds of S-series. Algebraic structures that we present do not depend on the particular problem or its dimension $d$ and are valid for any SDE with additive noise and can be generalized to SDEs with multiplicative noise in a straightforward way.

In Section \ref{sec:Coef_mult}, the formalism of exotic forests is used to define a theoretical algorithm that generates order conditions with respect to the invariant measure for numerical methods that can be expanded using B-series. The algorithm defines a linear map $A : \EF \to \widehat{\EF}$ that is applied to the truncated S-series corresponding to the method and returns a linear combination of exotic forests in which the coefficient of an exotic forest $\pi$ is denoted by $\omega (\pi)$. The order $p$ conditions obtained in this way have the form
\[ \omega (\pi) = 0 \quad \text{where } \pi \in \widehat{EF} \text{ with } |\pi| < p, \]
where $\widehat{EF}$ is a subset of exotic forests. We prove Theorem \ref{thm:Coef_mult}.

\vspace{0.2cm}
\noindent\textbf{Theorem \ref{thm:Coef_mult}}
\textit{
    Let $\cdot$ denote the concatenation product and let $\omega$ be the order condition map for a numerical method that can be expanded as a B-series over grafted trees, then,
    \begin{equation}\label{eq:Coef_mult}
        \omega(\pi_1 \cdot \pi_2) = \omega(\pi_1) \omega(\pi_2), \quad \text{for } \pi_1, \pi_2 \in EF.
    \end{equation}
}

Theorem \ref{thm:Coef_mult} allows us to decrease the number of order conditions with respect to the invariant measure for a class of numerical methods that includes stochastic Runge-Kutta methods defined in Definition \ref{RK_form}. If an exotic forest $\pi$ can be written as $\pi = \pi_1 \cdot \pi_2$, then the order condition $\omega(\pi) = 0$ is automatically satisfied if $\omega(\pi_1) = 0$ is satisfied. For example, Theorem \ref{thm:Coef_mult} implies the following relations between the order conditions: 
\[\begin{array}{cc}
    \omega (\forest{b,b}) = \omega(\forest{b})^2, & \omega (\forest{b,b,b}) = \omega(\forest{b})^3, \\
    \omega (\forest{b[b],b}) = \omega(\forest{b[b]}) \omega(\forest{b}), & \omega (\forest{b[1,1],b}) = \omega(\forest{b[1,1]}) \omega(\forest{b}),
\end{array}\]
which decreases the number of order conditions, in particular, for order $3$ from $13$ to $9$. The values of $\omega(\forest{b[b],b}), \omega(\forest{b[b]}),$ and $\omega(\forest{b})$ for stochastic Runge-Kutta methods with coefficients $b_i, a_{ij}, d_i$ with $i,j = 1,\dots,s$ are presented below,
\begin{align*}
    \omega (\forest{b[b]}) &= \sum_{i,j=1}^s b_i a_{ij} - \frac{1}{2} + \sum_{i=1}^s b_i - 2 \sum_{i=1}^s b_i d_i, \quad \quad \omega(\forest{b}) = \sum_{i=1}^s b_i - 1, \\
    \omega (\forest{b[b],b}) &= -2 \sum_{i,j=1}^s b_i d_i b_j - \frac{3}{2} \sum_{i=1}^s b_i + \sum_{i,j=1}^s b_i b_j + \sum_{i,j,k=1}^s b_i a_{ij} b_k + 2 \sum_{i=1}^s b_i d_i + \frac{1}{2} - \sum_{i,j=1}^s b_i a_{ij}.
\end{align*}
The list of values of $\omega$ for all exotic trees up to size $3$ can be found in Table \ref{table:append_ord_cond} in the Appendix.
This property was first observed for order $3$ by manual computation in \cite{Laurent_2019}. In this paper, we prove the property for arbitrary high order.

\section{Fundamentals}\label{sec:num_intro}

We consider the overdamped Langevin equation which is widely used in molecular dynamics and is ergodic under appropriate assumptions, 
\begin{equation}\label{eq:problem}
dX(t) = f(X(t))dt + \sqrt{2} dW(t),
\end{equation}
where $X(t) \in \mathbb{R}^d$ is a stochactic process with $X(0) = X_0$, the vector field $f = - \nabla V : \mathbb{R}^d \to \mathbb{R}^d$ is the gradient of a smooth and globally Lipschitz potential $V : \mathbb{R}^d \to \mathbb{R}$ that describes molecular interactions, $\sqrt{2}$ is a constant that can be changed by rescaling the problem in time, and $W(t)$ is a $d-$dimensional standard Wiener process fulfilling the usual assumptions.

Let us consider test functions $\phi \in \mathcal{C}^\infty_P (\mathbb{R}^d, \mathbb{R})$ which are taken to be smooth functionals on $\mathbb{R}^d$ with all partial derivatives having polynomial growth of the form 
\[ |\frac{\partial^n \phi(x)}{\partial x_{i_1} \cdots \partial x_{i_n}}| \leq C_n(1 + |x|^{s_n}) \quad \text{for any } n \in \mathbb{N} \text{ and } 1 \leq i_k \leq d \text{ with } k = 1, \dots, n, \]
with some constants $C_n$ and $s_n$ independent of $x$.
We consider numerical methods with the following weak Taylor expansion. Given an integrator $X_1 = \Psi_h (X_0, f, \xi)$, we have 
\[ \mathbb{E} [ \phi(X_1) | X_0 = x ] = \phi(x) + h \mathcal{A}_1 \phi(x) + h^2 \mathcal{A}_2 \phi(x) + \cdots, \]
where $\mathcal{A}_i$, $i = 1, 2, \dots$, are linear differential operators. For more details see \cite{Talay_Tubaro}. 

An integrator $X_1 = \Psi_h (X_0, f, \xi)$ satisfying the usual assumptions (see Section \ref{sec:Coef_mult}) has \textit{weak order} $q$ if
\[ |\mathbb{E} [\phi(X_1) | X_0 = x] - \mathbb{E} [\phi(X(h)) | X_0 = x] | = O(h^{q+1}), \]
where $\phi \in C^\infty_P (\mathbb{R}^d, \mathbb{R})$ is a test function. We note that the expectation of the functional of the exact solution has the following weak Taylor expansion:
\[ \mathbb{E} [ \phi(X(h)) | X_0 = x ] = \phi(x) + h \mathcal{L} \phi(x) + h^2 \frac{\mathcal{L}^2 \phi(x)}{2!} + \cdots + h^k \frac{\mathcal{L}^k \phi(x)}{k!} + \cdots, \]
with the generator $\mathcal{L} \phi := f \cdot \nabla \phi + \Delta \phi$ where $\Delta \phi = \sum_{i=1}^d \frac{\partial^2 \phi}{\partial x_i^2}$ denotes the Laplacian operator. Thus, an integrator has weak order $q$ if 
\[ \mathcal{A}_k = \frac{\mathcal{L}^k}{k!}, \quad \text{for } k = 1, \dots, q.  \]

In this paper, we consider ergodic problems that have unique invariant measures that characterize the trajectories of the system. 
\begin{definition}
	A problem is \textit{ergodic} if there exists a unique invariant measure $\mu$ satisfying for all deterministic initial conditions $X_0$ and all smooth test functions $\phi$,
	\[ \lim_{T \to \infty} {1 \over T} \int_0^T \phi (X(s)) ds = \int_{\mathbb{R}^d} \phi (x) d \mu (x), \quad \text{almost surely.} \]
\end{definition}
A similar definition can be applied to numerical integrators. 
\begin{definition}
	A numerical method $X_1 = \Psi_h (X_0, f, \xi)$ is ergodic if there exists a unique invariant probability law $\mu^h$ with finite moments of any order satisfying for all deterministic initial conditions $X_0 = x$ and all smooth test functions $\phi$,
	\[ \lim_{N \to \infty} {1 \over N + 1} \sum_{n=0}^N  \phi (X_n) = \int_{\mathbb{R}^d} \phi (x) d \mu^h (x), \quad \text{almost surely.}\]
	See \cite{Mattingly_2002} for more details.
\end{definition}

Integrators can be used to approximate the invariant measure $\mu$ of the system using the invariant measure $\mu^h$ of the integrator. The accuracy of the approximation is characterized by the order of the integrator with respect to the invariant measure.

\begin{definition}
	A numerical method $X_1 = \Psi_h (X_0, f, \xi)$ has order $p$ with respect to the invariant measure of the SDE if
	\[ \left| \int_{\mathbb{R}^d} \phi (x) d \mu^h (x) - \int_{\mathbb{R}^d} \phi (x) d \mu (x) \right| \leq C h^p, \]
	where $C$ is independent of $h$ assumed small enough.
\end{definition}

We note that an order $p$ with respect to the invariant measure can be shown for a large class of integrators using the weak Taylor expansion. The details are discussed in Section \ref{sec:num_framework}. We also note that the order $p$ with respect to the invariant measure is higher or equal than the weak order $q$ of the integrator, that is $p \geq q$. We will use the following form of stochastic Runge-Kutta methods.

\begin{definition}\label{RK_form}
    Let $a_{ij}, b_i, d^{(k)}_i$ be the coefficients defining the stochastic Runge-Kutta (sRK) scheme, and $\xi_n^{(k)} \sim \mathcal{N} (0, I_d)$ be independent Gaussian random vectors. Then, the stochastic Runge-Kutta scheme has the form:
\begin{align*}
Y_i &= X_n + h \sum_{j=1}^s a_{ij} f(Y_j) + \sum_{k=1}^l d_i^{(k)} \sqrt{h} \sqrt{2} \xi_n^{(k)}, \quad i = 1, \dots, s, \\
X_{n+1} &= X_n + h \sum_{i=1}^s b_i f(Y_i) + \sqrt{h} \sqrt{2} \xi_n^{(1)}.
\end{align*}
\end{definition}
We shall assume for simplicity of the presentation that $l = 1$ which is sufficient to achieve weak order $2$ or order $3$ with respect to the invariant measure. We note that $l > 1$ is necessary in general to achieve high order~\cite{Laurent_2019}. The analysis in this paper extends naturally to the $l > 1$ case by considering grafted forests with decorated grafted vertices. Two grafted vertices can form a pair only if they are decorated by the same number. For example, for $l = 2$, we should consider grafted of the form
\[ \forest{b[x_1_90,x_1_90]}, \quad \forest{b[b[x_1_90,x_2_90],x_1_45],x_2_45}. \]
More details on grafted and exotic forests can be found in Section \ref{sec:tree_formalism}.

\subsection{The framework of B-series and S-series}

Let us consider the space $\mathcal{X}$ of vector fields on $\mathbb{R}^d$. Let $f, g \in \mathcal{X}$ and let $f[g]$ denote the differentiation of $g$ in the direction of $f$, that is, for $p \in \mathbb{R}^d$, we have
\[ f[g] (p) = \big( f(p)[g] \big) (p) = \sum_{i=1}^d f^i (p) \partial_i g (p), \quad \text{with } \partial_i g := \frac{\partial g}{\partial x_i}. \]
This way, vector fields define differential operators of degree one. The differential operators of higher degrees can be obtained by pointwise composition of vector fields, for example, let $f, g, h \in \mathcal{X}$ and $\partial_{ij} := \partial_i \partial_j$, then, for $p \in \mathbb{R}^d$, we have
\[ (f g) [h] (p) = \big( (f(p) g(p)) [h] \big) (p) = \sum_{i, j = 1}^d f^i (p) g^j (p) \partial_{ij} h (p), \]
From now on we will omit writting $p$ and the differentiation will be written as
\[ f[g] = \sum_{i=1}^d f^i \partial_i g, \quad \text{and} \quad f g [h] = \sum_{i,j=1}^d f^i g^j \partial_{ij} h. \]
Due to the fact that the pointwise composition is commutative, differentiation is a pre-Lie product, that is, it satisfies the following relation
\[ f [g [h]] - f[g][h] = g [f [h]] - g[f][h], \quad \text{for } f, g, h \in \mathcal{X}. \]
We consider an initial value ODE of the form
\begin{equation}\label{eq:classic_problem}
    \frac{dy}{dt} = f(y), \quad y(0) = y_0. 
\end{equation}
The elementary differentials that appear as terms in the Taylor expansion of $y(h)$ around $0$ form a pre-Lie algebra with the product given by differentiation. For example,
\[ hf [h^2 f^\prime f] = h^3 f^\prime f^\prime f + h^3 f^{\prime \prime} (f, f). \]

Let us consider the pre-Lie algebra of non-planar rooted trees $(\mathcal{T}, \graft)$ with $\graft$ being the grafting product on trees defined by attaching the root of the left operand to a vertex of the right operand in all possible ways, for example,
\[ \forest{b} \graft \forest{b[b]} = \forest{b[b[b]]} + \forest{b[b,b]}. \]
We extend the grafting product to the commutative algebra of forests $(\mathcal{F}, \cdot)$ which is the symmetric algebra on trees, $(\mathcal{F}, \cdot) := S_\mathbb{R} (\mathcal{T})$. Let $\tau \in \mathcal{T}$ and $\pi_1, \pi_2 \in \mathcal{F}$, then
\[ (\tau \cdot \pi_1) \graft \pi_2 = \tau \graft (\pi_1 \graft \pi_2) - (\tau \graft \pi_1) \graft \pi_2, \]
\[ \tau \graft (\pi_1 \cdot \pi_2) = (\tau \graft \pi_1) \cdot \pi_2 + \pi_1 \cdot (\tau \graft \pi_2). \]
We note that this definition of grafting on forests is well-defined since grafting is a pre-Lie product. The details can be found in \cite{OudomGuin}. 

\begin{remark}
    The algebra $(\mathcal{F}, \cdot, \graft)$ forms a commutative version of D-algebra structure introduced in \cite{MuntheKaas08oth}.
\end{remark}
In \cite{Chapoton_01}, it is proven that the algebra $(\mathcal{T}_n, \graft)$ with $n-$colored trees is the free pre-Lie algebra with $n$ generators. Therefore, there exists a surjective morphism from the pre-Lie algebra $(\mathcal{T}, \graft)$ onto the pre-Lie algebra of elementary differentials generated by $hf$. The morphism is extended to $(\mathcal{F}, \cdot, \graft)$ by sending the commutative product to the pointwise composition product of vector fields. The morphism is denoted by $F_f$ and we give the explicit formula in Definition \ref{def:F_f_classic}. Let us use the following notation, $[d] := \{ 1, \dots, d \}$.

\begin{definition}\label{def:F_f_classic}
    Let $\tau$ be a tree with all edges being directed towards the root, $V(\tau)$ be the set of vertices of $\tau$, and $p(v)$ be the set of predecessors of $v$ in $\tau$. Then,
    \[ F_f(\tau)^k = \sum_{\substack{\alpha: V(\tau) \to [d]\\\alpha(r) = k}} \prod_{v \in V(\tau)} (\prod_{u \in p(v)} \partial_{\alpha(u)}) F_f (v)^{\alpha(v)}, \]
    where the sum is taken over all decorations of $\tau$ by the set $[d]$ such that the root $r$ of $\tau$ is decorated by $k \in [d]$, and $F_f (\bullet) = hf$ where $h$ is the timestep.
\end{definition}
For example, $F_f (\forest{b[b,b]}) = h^3 \sum_{i, j = 1}^d f^i f^j \partial_{ij} f$, where we decorate the tree as $\forest{b_k[b_i,b_j]}$ for the $k^{th}$ component with $i,j,k \in 1, \dots, d$.

\begin{definition}\cite{Butcher_63,Hairer_Wanner}
    B-series are formal sums of vector fields of the following form
    \[ B(a) = \sum_{\tau \in T} \frac{a(\tau)}{\sigma(\tau)} F_f(\tau), \]
    where $T$ is the set of rooted non-planar trees, $a : T \to \mathbb{R}$ is a functional, and $\sigma(\tau)$ is the size of the automorphism group of $\tau$. 
\end{definition}

The exact solution $y(h)$ and one-step of a Runge-Kutta method $\Psi_h (A, b, f)$ can be expanded using B-series as $y_0 \mapsto y_0 + B(a)(y_0)$ with the functionals $a : T \to \mathbb{R}$ defined appropriately.
The concept of S-series was used to study the first integrals of B-series~\cite{Mur99}. Let $I : \mathbb{R}^d \to \mathbb{R}$ be a first integral, then we have the following property
\[ I\big(y_0 + B(a)(y_0)\big) = S(a)[I](y_0) = \sum_{\pi \in F} \frac{a(\pi)}{\sigma(\pi)} F_f(\pi)[I](y_0), \]
where $S(a)$ is called an S-series, $F$ is the set of forests, the functional $a : F \to \mathbb{R}$ extends to forests by $a(\pi_1 \cdot \pi_2) = a(\pi_1) a(\pi_2)$ for $\pi_1, \pi_2 \in F$. We note that $y_0 + B(a)(y_0) = S(a)[\text{Id}](y_0)$ where $\text{Id}$ is the identity $\text{Id} (x) = x$. Similar ideas are used to write the flow of a differential equation as the exponential of $F_f (\bullet)$, i.e. 
\[ y(h) = \exp\big(F_f(\bullet)\big) \text{Id}(y_0) = S(\alpha)[\text{Id}](y_0), \]
where $\alpha$ is an appropriate functional on forests. The details can be found in Chapter III.5.1 of \cite{GeomIntBook} in the context of the Baker-Campbell-Hausdorff formula for splitting integrators.
We note that in a stochastic context, we can replace the first integral $I$ with a test function $\phi$ and use S-series to study the expectation of a functional of one-step of a numerical integrator, i.e. $\mathbb{E}[\phi (y_0 + B(a)(y_0))]$, using its weak Taylor expansion~\cite{Talay_Tubaro}.

An important feature of B-series and S-series is that they are completely characterized by the functionals $a : T \to \mathbb{R}$. This allows us to use combinatorial properties and algebraic structures on trees and forests to study the properties and operations of numerical integrators.

\subsection{Extended tree formalism}\label{sec:tree_formalism}

We extend the framework of B-series and S-series by extending the tree and forest formalisms. We consider the sets of aromatic trees $AT = A \times T$ where $A$ is the set of multi-aromas, i.e. graphs in which every vertex has exactly one outgoing edge. The set $A$ includes the empty graph and some of its elements are
\[ \mathbf{1}, \quad \forest{(b)}, \quad \forest{(b[b]),(b,b,b[b])}, \quad \forest{(b[b,b]),(b[b])}, \quad \forest{(b[b],b,b,b[b,b[b]]),(b[b],b)}. \]
The corresponding vector spaces are denoted by $\A$ and $\AT$, respectively. We obtain the set of aromatic forests $AF$ by concatenating aromatic trees in all possible ways, including the empty forest $\mathbf{1}$. The corresponding vector space is denoted by $\AF$.

We define decorated aromatic forests as aromatic forests $\pi$ together with maps $\alpha : V(\pi) \to D$ that send vertices of $\pi$ to decorations from the set $D$ which is defined depending on the type of forests we want to represent. The set of decorated aromatic forests with an abstract set $D$ is denoted by $AF_D$ and the vector space by $\AF_D$. 

Let us consider the space of bicolored aromatic forests $\AF_{\bullet, \times}$ spanned by aromatic forests $\pi \in AF$ together with decorations $\alpha_g : V(\pi) \to \{ \bullet, \times \}$. Let the space of grafted forests be defined as the quotient space $\AF_g := \AF_{\bullet, \times} /_{\mathcal{K}_\times}$ with
\[ \mathcal{K}_\times := \text{span} \{ (\pi, \alpha_g) \in AF_{\bullet, \times} \; : \; \exists (v, u) \in E(\pi), \, \alpha_g(u) = \times \}. \]
That is, grafted forests are bicolored aromatic forests $(\pi, \alpha_g)$ for which $\alpha_g^{-1} (\times)$ is a subset of leaves of $\pi$.
For example, some grafted trees are listed below
\[ \forest{(b,b),b}, \quad \forest{(b[b]),b}, \quad \forest{(b),b[b]}, \quad \forest{(b[x,x]),b}, \quad \forest{(b[x]),b[x]}, \quad \forest{(b),b[x,x]}, \quad \forest{(b,b[x]),x}, \quad \forest{(b[b,x]),x}, \quad \forest{(b[b[x]]),x}, \quad \forest{(b[x,x,x]),x}. \]
The size of a grafted forest is taken to be the sum of weights of vertices with black vertices having weight $1$ and grafted vertices having weight $0.5$. All grafted trees up to size $3$ are listed in Table \ref{table:all_grafted_exotic_trees} of the Appendix. Grafted forests arise when we consider the overdamped Langevin equation and the B-series \cite{Laurent_2019} that are used to study it. The sets of grafted forests and trees are denoted by $AF_g$ and $AT_g$, and the corresponding vector spaces are denoted by $\AF_g$ and $\AT_g$.

Exotic forests are grafted forests with even number of grafted vertices in which all grafted vertices are paired. For example, some exotic trees are listed below
\[ \forest{(b,b),b}, \quad \forest{(b[b]),b}, \quad \forest{(b),b[b]}, \quad \forest{(b[1,1]),b}, \quad \forest{(b[1]),b[1]}, \quad \forest{(b),b[1,1]}, \quad \forest{(b,b[1]),1}, \quad \forest{(b[b,1]),1}, \quad \forest{(b[b[1]]),1}, \quad \forest{(b[1,1,2]),2}. \]
All exotic trees up to size $3$ are listed in Table \ref{table:all_grafted_exotic_trees} of the Appendix.
We note that the pairing between two grafted vertices is denoted by associating a natural number to the two grafted vertices. The choice of the particular natural number does not matter. Exotic forests are used to represent the differential operators that appear in the expansion of $\mathbb{E}[\phi(y_0 + B(a)(y_0))]$ applied to the overdamped Langevin equation~\cite{Laurent_2019}. We build the space of exotic forests in several steps. 
\begin{enumerate}
    \item consider a space $\AF_{\bullet \mathbb{N}}$ spanned by the set $AF_{\bullet \mathbb{N}}$ of aromatic forests $\pi$ with decorations $\alpha : V(\pi) \to \{\bullet\} \sqcup \mathbb{N}$ such that $|\alpha^{-1}(k)|$ is even for all $k \in \mathbb{N}$, 
    \item let $\widetilde{\AF}_{\bullet \mathbb{N}}$ be the completion with respect to the graduation given by the number of vertices. The elements of $\widetilde{\AF}_{\bullet \mathbb{N}}$ are formal sums of the form
        \[ \sum_{(\pi, \alpha) \in AF_{\bullet \mathbb{N}}} a(\pi, \alpha) (\pi, \alpha), \quad \text{for } a \in \AF_{\bullet \mathbb{N}}^*, \]
    \item define the space of exotic forests $\EF$ to be spanned by the set $EF$ of elements
        \begin{equation}\label{eq:exotic_def}
            (\pi, \alpha_e) := \sum_{\alpha \in P(\alpha_e)} (\pi, \alpha) \  \in \  \widetilde{\AF}_{\bullet \mathbb{N}},
        \end{equation}
        with $\alpha_e$ being a decoration of $\pi \in AF$ by $\{\bullet\} \sqcup \mathbb{N}$ such that $|\alpha^{-1}(k)| = 2$ for all $k \in \mathbb{N}$ and $P(\alpha_e)$ is the set of decorations $\alpha$ with $\alpha^{-1}(\bullet) = \alpha_e^{-1}(\bullet)$ and 
        \[ \alpha(v_1) = \alpha(v_2) \quad \text{if} \quad \alpha_e(v_1) = \alpha_e(v_2), \quad \text{for } v_1, v_2 \in V(\pi). \]
\end{enumerate}
We say that if $P(\alpha_{e,1}) = P(\alpha_{e,2})$, then $\alpha_{e,1} = \alpha_{e,2}$. For example, the following two exotic forest $(\pi_1, \alpha_{e,1})$ and $(\pi_2, \alpha_{e,2})$ are equal
\[ \forest{(b[b[3],1,1]),b[b[2],b[2,3]]} = \forest{(b[b[2],3,3]),b[b[1],b[1,2]]}, \quad \text{since } P(\alpha_{e,1}) = P(\alpha_{e,2}). \]
A pair of grafted vertices forms a \textit{liana}. This terminology was proposed in \cite{Laurent_2019}\textcolor{blue}{,} in which the pairs of grafted vertices are replaced by dotted lines.

\begin{definition}\label{def:AF_D_morphism}
    A morphism $\varphi : (\pi_1, \alpha_1) \to (\pi_2, \alpha_2)$ between two decorated aromatic forests is a morphism between the aromatic forests $\varphi : \pi_1 \to \pi_2$ such that $\alpha_1 = \alpha_2 \circ \varphi$. 
\end{definition}
An isomorphism is an invertible morphism, and an automorphism is an isomorphism of a decorated aromatic forest with itself. The group of all automorphisms of a decorated aromatic forests $(\pi, \alpha)$ is denoted by $\Aut (\pi, \alpha)$. The symmetry coefficient of $(\pi, \alpha)$ is denoted by $\sigma (\pi, \alpha)$ and is defined as the size of the automorphism group. For example,
\[ \sigma(\forest{b[x,x,b[x,b]]}) = 2, \quad \sigma(\forest{x,b[x,x,b[x]]}) = 2, \quad \sigma(\forest{b[x,x,x,x]}) = 4!, \quad \sigma(\forest{1,b[1,2,b[2]]}) = 1, \quad \sigma(\forest{b[1,1,2,2]}) = 8. \]

The combinatorial and algebraic structures of decorated aromatic forests are studied in more detail in Section \ref{sec:Asigma}. We note that we will often omit writting $\alpha$ and denote the elements of $AF_D$ simply by $\pi$. In the cases when $\pi \in AF_g$ or $\pi \in EF$, the corresponding decorations are denoted by $\alpha_g$ and $\alpha_e$, respectively.

We extend the concatenation product of forests to the concatenation product of decorated aromatic forests as follows. Let $\pi_i \in AF$ and $\alpha_i : V(\pi_i) \to D$ for $i = 1,2$, then,
\[ (\pi_1, \alpha_1) \cdot (\pi_2, \alpha_2) = (\pi_1 \cdot \pi_2, \alpha_1 \sqcup \alpha_2), \]
where $\alpha_1 \sqcup \alpha_2 : V(\pi_1 \cdot \pi_2) \to D$ with $(\alpha_1 \sqcup \alpha) \restrict{V(\pi_i)} = \alpha_i$. We note that the subspace $\mathcal{K}_\times$ forms an ideal with respect to the concatenation product, therefore, it is extended to grafted and exotic forests. We extend the grafting product to the decorated aromatic forests and, by the same argument, to grafted and exotic forests.

We recall that forests, including grafted and exotic forests, are used to represent differential operators with the grafting product representing the differentiation. Let us consider the Grossman-Larson product denoted by $\gl$, which represents the composition of differential operators, that is,
\[ \pi_1 \graft (\pi_2 \graft \cdot) = (\pi_1 \gl \pi_2) \graft \cdot, \quad \text{with } \pi_1, \pi_2 \in AF_D. \]

Let us now define the Connes-Kreimer coproduct \cite{Bogfjellmo} on decorated aromatic forests.
\begin{definition}\label{def:CK}
	The Connes-Kreimer coproduct on $AF_D$ is defined as
    \[ \Delta_{CK} (\pi, \alpha) := \sum_{\pi_0 \subset \pi} (\pi \setminus \pi_0, \alpha \restrict{\pi \setminus \pi_0}) \otimes (\pi_0, \alpha \restrict{\pi_0}), \]
	where the sum runs over all rooted subforests $\pi_0 \in AF$ of $\pi$ such that $\pi \setminus \pi_0 \in AF$ and there are no edges going from $\pi_0$ to $\pi \setminus \pi_0$ in $\pi$.
\end{definition}
We recall that aromatic forests that we consider must have at least one root. This differs with \cite{Bogfjellmo} where multi-aromas are also included in $AF$. For example,
\begin{align*}
    \Delta_{CK} (\forest{(b[b]),b[b[x]]}) = &\mathbf{1} \otimes \forest{(b[b]),b[b[x]]} \  + \  \forest{b} \otimes \forest{(b),b[b[x]]} \  + \  \forest{x} \otimes \forest{(b[b]),b[b]} \  + \  \forest{b[x]} \otimes \forest{(b[b]),b} \  + \  \\
    & \forest{b,x} \otimes \forest{(b),b[b]} \  + \  \forest{b,b[x]} \otimes \forest{(b),b} \  + \  \forest{(b[b]),x} \otimes \forest{b[b]} \  + \  \forest{(b[b]),b[x]} \otimes \forest{b} \  + \  \forest{(b[b]),b[b[x]]} \otimes \mathbf{1}.
\end{align*}
We denote the dual of the Connes-Kreimer coproduct by $\ckdual$ and call it dual CK product. We extend the Grossman-Larson and dual CK products to decorated aromatic forests including grafted and exotic forests. We note that to define the Connes-Kreimer coproduct on grafted and exotic forests, we need to dualize their constructions, that is, interchange taking a subspace and taking a quotient by its complement.

\section{Combinatorial algebra framework of decorated aromatic forests}\label{sec:Asigma}

 We use decorated aromatic forests to prove combinatorial relations that are essential in our description of the algebraic structure of exotic S-series. We introduce a relationship between decorations and use it to prove the relationship between the Grossman-Larson and dual CK products on decorated aromatic forests. This is also used to describe the way exotic S-series are obtained by taking the expectation of S-series over grafted forests. 

The relationship between Grossman-Larson and dual CK product on aromatic forests was proven implicitly in \cite{Bogfjellmo} and on classical forests in \cite{Hoffman}. We present an alternative proof which is easily generalizable to more complex sets of forest, for example, grafted and exotic forests.

\subsection{Relationship between two decorations}\label{sec:rel_two_decorations}

\begin{definition}\label{def:finer_decoration}
    Let $\alpha : V(\pi) \to D$ and $\hat{\alpha} : V(\pi) \to \hat{D}$ be two decorations of an aromatic forest $\pi \in AF$. Decoration $\alpha$ is said to be \textit{finer} than $\hat{\alpha}$ if there exists a surjective map $\Phi : D \to \hat{D}$ such that $\hat{\alpha} = \Phi \circ \alpha$. 
\end{definition}
For example, let $\Phi : D_e \to D_g$ for $D_e = \{\bullet\} \sqcup \N$ and $D_g = \{\bullet, \times\}$ be the map defined as $\Phi(\bullet) := \bullet$ and $\Phi(k) = \times$ for all $k \in \N$, then, it induces a map $\Phi_\pi : (\pi, \alpha_e) \to (\pi, \alpha_g)$, e.g.
\[ \Phi_\pi (\forest{b[1,2,b[1,2]]}) = \forest{b[x,x,b[x,x]]}. \]
We note that $\Phi_\pi$ is well-defined on the equivalence classes $(\pi, \alpha_e)$ that are used to define the exotic forests. We say that the decoration of the exotic forests is finer than the decoration of the grafted forests.

\begin{definition}
    Let $p(\pi, \alpha, \hat{\alpha})$, with $\alpha$ being finer than $\hat{\alpha}$, denote the number of decorations $\tilde{\alpha}$ such that $(\pi, \tilde{\alpha}) \cong (\pi, \alpha)$ and $ \hat{\alpha} = \Phi \circ \tilde{\alpha} $ where $\Phi$ is the map such that $ \hat{\alpha} = \Phi \circ \alpha $. 
\end{definition}
If $(\pi, \alpha_e) \in EF$ is an exotic forest, then $p(\pi, \alpha_e, \alpha_g)$ is the number of ways to pair grafted vertices of $(\pi, \alpha_g) \in F_g$ to obtain a forest isomorphic to $(\pi, \alpha_e)$. 

\begin{example}
    Let us consider 
    \[ (\pi, \alpha_e) = \forest{b[1,2,b[1,2]]}, \quad (\pi, \alpha_g) = \forest{b[x,x,b[x,x]]}, \quad \text{with } \pi = \forest{b[b,b,b[b,b]]}, \]
    and find the value of $p(\pi, \alpha_e, \alpha_g)$. First, we have to find the map $\Phi$ such that $\alpha_g = \Phi \circ \alpha_e$. The map $\Phi$ is defined in the following way: $\Phi(k) := \times, \Phi(\bullet) = \bullet$ for $k \in \N$. Next, let us list the elements of the equivalence class $(\pi, \alpha_e)$: 
    \[ (\pi, \alpha_e) = \{ \forest{b[i_i,i_k,b[i_i,i_k]]} \; : \; i, k \in \N \}, \]
    and notice that the only choice for $\tilde{\alpha}_e \neq \alpha_e$ such that $(\pi, \tilde{\alpha}_e) \cong (\pi, \alpha_e)$ and $\Phi \circ \tilde{\alpha}_e = \alpha_g$ is
    \[ (\pi, \tilde{\alpha}_e) = \{ \forest{b[i_k,i_i,b[i_i,i_k]]} \; : \; i, k \in \N \}, \]
    Therefore, $p(\pi, \alpha_e, \alpha_g) = 2$.
\end{example}

\begin{proposition}\label{prop:p_generalized}
Let $\pi \in AF$ be an aromatic forest and let $\alpha : V(\pi) \to D$ be finer than $\hat{\alpha} : V(\pi) \to \hat{D}$, then
\[ p(\pi, \alpha, \hat{\alpha}) = \frac{\sigma(\pi, \hat{\alpha})}{\sigma(\pi, \alpha)}. \]
\end{proposition}
\begin{proof}
We note that there exists a map $\Phi : D \to \hat{D}$ such that $\hat{\alpha} = \Phi \circ \alpha$. We note that $\Phi$ is far from being injective and, generally, there are multiple structures $\tilde{\alpha}$ that give $\hat{\alpha}$ when composed with $\Phi$ on the left. By the definition of $p(\pi, \alpha, \hat{\alpha})$, we have
\[ p(\pi, \alpha, \hat{\alpha}) = |\{ \tilde{\alpha} : V(\pi) \to D \; : \; \Phi \circ \tilde{\alpha} = \hat{\alpha} , \; (\pi, \tilde{\alpha}) \cong (\pi, \alpha) \}| =: |A|. \]
    To prove the statement we have to count the elements of the set $A$. Every element $\tilde{\alpha}$ of $A$ forms with $\pi$ a forest isomorphic to $(\pi, \alpha)$, that is, there exists $\varphi : (\pi, \tilde{\alpha}) \to (\pi, \alpha)$ such that $\tilde{\alpha} = \alpha \circ \varphi$. If we compose both sides of the equality with $\Phi$, we get
\[ \Phi \circ \tilde{\alpha} = \Phi \circ \alpha \circ \varphi, \]
    which implies that $\varphi$ is in $\Aut(\pi, \hat{\alpha})$ since $\hat{\alpha} = \Phi \circ \alpha = \Phi \circ \tilde{\alpha}$. Therefore, $A \subset \Aut(\pi, \hat{\alpha})$. Take $\varphi \in \Aut(\pi, \hat{\alpha})$, we know that $\hat{\alpha} = \hat{\alpha} \circ \varphi$. Use the fact that $\hat{\alpha} = \Phi \circ \alpha$ to get
\[ \hat{\alpha} = \Phi \circ \alpha \circ \varphi. \]
We let $\tilde{\alpha} = \alpha \circ \varphi$ and notice that $\tilde{\alpha} = \alpha$ if and only if $\varphi \in \Aut(\pi, \alpha)$. Therefore, 
\[|A| = |\Aut(\pi, \hat{\alpha})/_{\Aut(\pi, \alpha)}| = \frac{\sigma(\pi, \hat{\alpha})}{\sigma(\pi, \alpha)}, \]
and the statement is proved.
\end{proof}

\subsection{Relationship between the two products}

Let the map $A_\sigma : \AF_D \to \AF_D$ be defined as $A_{\sigma} (\pi, \alpha) = \sigma(\pi, \alpha) (\pi, \alpha)$. The map $A_\sigma$ sends a decorated aromatic forest to itself multiplied by its symmetry. We prove that the linear map $A_\sigma$ induces algebra isomorphism $A_\sigma : (\AF_D, \gl) \to (\AF_D, \ckdual)$. 

We note that an aromatic forest can be decorated by multiple sets by taking their Cartesian product. For example,
\[ (\pi, \alpha_g, \alpha_\N) := (\pi, \alpha_g \times \alpha_\N) = \forest{b_7[x_2_180,b_2[x_4]]}, \quad \text{where } \pi = \forest{b[b,b[b]]}, \quad \alpha_g \times \alpha_\N: V(\pi) \to D_g \times \N. \]

We define a set of labeled decorated aromatic forests $AF_{DL}$ and the corresponding space $\AF_{DL}$. We note that labels, unlike decorations, are required to be in bijection with the vertices. The elements $(\pi, \alpha, \alpha_l)$ of $AF_{DL}$ are aromatic forests $\pi \in AF$ decorated by $\alpha : V(\pi) \to D$ for some set $D$ and $\alpha_l : V(\pi) \to \N$ such that $\alpha_l$ is an injection. 

Let us define the Grossman-Larson and dual CK products on the space $\AF_{DL}$ by considering the space $\AF_{D\N}$ of elements $(\pi, \alpha, \alpha_{\N})$ where $(\pi, \alpha) \in AF_D$ and $\alpha_\N : V(\pi) \to \N$. Let us define a subspace $\mathcal{K}$ which is an ideal in $(\AF_{D\N}, \gl)$ and $(\AF_{D\N}, \ckdual)$:
\[ \mathcal{K} := \text{span} (K) = \text{span} \{ (\pi, \alpha, \alpha_\N) \in AF_{D\N} \; : \; \exists k \in \N \text{ such that } |\alpha_\N^{-1}(k)| > 1 \}. \]
Then, the corresponding algebras $(\AF_{DL}, \gl)$ and $(\AF_{DL}, \ckdual)$ are defined as
\[ (\AF_{DL}, \gl) := (\AF_{D\N}, \gl) /_{\mathcal{K}}, \quad (\AF_{DL}, \ckdual) := (\AF_{D\N}, \ckdual) /_{\mathcal{K}}. \]
Then, the Grossman-Larson product and the product dual to the Connes-Kreimer coproduct are identical on the space $\AF_{DL}$. For example,
\[ \forest{t_1[t_2]} \gl \forest{t_3[t_4,t_5]} = \forest{t_1[t_2]} \ckdual \forest{t_3[t_4,t_5]} = \forest{t_1[t_2],t_3[t_4,t_5]} + \forest{t_3[t_1[t_2],t_4,t_5]} + \forest{t_3[t_4[t_1[t_2]],t_5]} + \forest{t_3[t_4,t_5[t_1[t_2]]]}, \]
\[ \forest{t_1} \gl \forest{t_3[t_1]} = \forest{t_1} \ckdual \forest{t_3[t_1]} = 0. \]

Let us denote by $\widetilde{\AF}_{DL}$ the completion of the space $\AF_{DL}$ with respect to the graduation given by the number of vertices in the aromatic forests. That is, $\widetilde{\AF}_{DL}$ is a space of formal sums of the form
\[ \sum_{\pi \in AF_{DL}} a(\pi) \pi, \quad \text{with } a \in \AF_{DL}^*.  \]

\begin{proposition}\label{prop:AF_into_AFuD}
    Define the maps $\varphi$ and $\hat{\varphi}$ as 
    \begin{align*}
        \varphi : (\AF_D, \gl) \to (\widetilde{\AF}_{DL}, \gl), \quad &\varphi(\pi, \alpha) := \sum_{\alpha_l} (\pi, \alpha, \alpha_l), \quad \text{and} \\
        \hat{\varphi} : (\AF_D, \ckdual) \to (\widetilde{\AF}_{DL}, \ckdual), \quad &\hat{\varphi}(\pi, \alpha) := \sum_{[\alpha_l]} (\pi, \alpha, \alpha_l), 
    \end{align*}
    where $\varphi(\pi, \alpha)$ is the sum are over all labelings $\alpha_l$ and $\hat{\varphi}(\pi, \alpha)$ is the sum are over all labelings $\alpha_l$ such that all terms of the sum are distinct. Then, the maps $\varphi$ and $\hat{\varphi}$ are injective algebra morphisms.
\end{proposition}
\begin{example}
    Let $\varphi$ and $\hat{\varphi}$ be the maps defined in Proposition \ref{prop:AF_into_AFuD}, then,
    \[ \varphi(\forest{b[b,b]}) = \sum_{k, i, j \in \mathbb{N}} \forest{t_k[t_i,t_j]}, \quad \hat\varphi(\forest{b[b,b]}) = \sum_{k, i < j \in \mathbb{N}} \forest{t_k[t_i,t_j]}. \]
\end{example}
\begin{proof}
    We can see that the map $\varphi : (\AF_D, \gl) \to (\widetilde{\AF}_{DL}, \gl)$ is indeed an injective morphism. To prove that the map $\hat{\varphi} : (\AF_D, \ckdual) \to (\widetilde{\AF}_{DL}, \ckdual)$ is an injective morphism, we consider its dual. The dual of $\hat{\varphi}$ is the map $\hat{\varphi}^* : (\widetilde{\AF}_{DL}, \Delta_{CK}) \to (\AF_D, \Delta_{CK})$ with $\hat{\varphi}^* (\pi, \alpha, \alpha_l) = (\pi, \alpha)$. We see that the dual is a surjective coalgebra morphism. Therefore, $\hat{\varphi}$ is an injective algebra morphism.
\end{proof}

\begin{proposition}\label{prop:Asigma_general}
    The map $A_\sigma : (\AF_D, \gl) \to (\AF_D, \ckdual)$ is an algebra isomorphism.
\end{proposition}
\begin{proof}
    Let us consider the maps $\varphi : (\AF_D, \gl) \to (\widetilde{\AF}_{DL}, \gl)$ and $\hat{\varphi} : (\AF_D, \ckdual) \to (\widetilde{\AF}_{DL}, \ckdual)$ from the Proposition \ref{prop:AF_into_AFuD}. We show that $\varphi = \hat{\varphi} \circ A_\sigma$ and use this fact as a key ingredient of the proof. We have
    \[ \varphi (\pi, \alpha) = \sum_{\alpha_l} (\pi, \alpha, \alpha_l) = \sum_{[\alpha_l]} p(\pi, \alpha, \alpha \times \alpha_l) (\pi, \alpha, \alpha_l) = \sum_{[\alpha_l]} \frac{\sigma(\pi, \alpha)}{\sigma(\pi, \alpha, \alpha_l)} (\pi, \alpha, \alpha_l), \]
    where $p(\pi, \alpha, \alpha \times \alpha_l)$ is the number of ways to obtain $(\pi, \alpha, \alpha_l)$ from $(\pi, \alpha)$. We use Proposition \ref{prop:p_generalized} and we note that $\sigma(\pi, \alpha, \alpha_l) = 1$ due to the definition of a labeling. Therefore,
    \[ \varphi (\pi, \alpha) = \sum_{[\alpha_l]} \sigma(\pi, \alpha) (\pi, \alpha, \alpha_l) = (\hat{\varphi} \circ A_\sigma) (\pi, \alpha). \]
    We use this property together with the fact that $\gl = \ckdual$ in $\AF_{DL}$ to show that
    \[ \varphi \circ \gl = \ckdual \circ (\varphi \otimes \varphi) = \hat{\varphi} \circ \ckdual \circ (A_\sigma \otimes A_\sigma) = \varphi \circ A^{-1}_\sigma \circ \ckdual \circ (A_\sigma \otimes A_\sigma). \]
    We use the injectivity of $\varphi$ to finish the proof.
\end{proof}

We note that Propoposition \ref{prop:Asigma_general} is proved for any Grossman-Larson and dual CK algebras over decorated aromatic forests, including the cases $\AF_D = \AF_g$ and $\AF_D = \EF$, that is, of grafted and exotic forests.

\section{Exotic S-series}\label{sec:exotic_S-series}

Let us introduce S-series over decorated aromatic forests. We denote by $\widetilde{\AF}_D$ the space of formal sums of the form
\[ \sum_{\pi \in AF_D} a(\pi) \pi, \quad \text{with } a \in \AF_D^*.  \]
It is the completion with respect to the graduation given by the number of vertices.
Let $\delta_\sigma : \AF_D^* \to \widetilde{\AF}_D$ be the isomorphism between the dual and the completion given by
\[ \delta_\sigma (a) = \sum_{\pi \in AF_D} \frac{a(\pi)}{\sigma(\pi)} \pi. \]
We assume that the map $F_f$ (Definition \ref{def:F_f_classic}) is defined over decorated aromatic forests and is a Grossman-Larson algebra morphism,
\[ F_f(\pi_1 \gl \pi_2) [\cdot] = F_f(\pi_1) \big[F_f(\pi_2) [\cdot]\big], \quad \text{for } \pi_1, \pi_2 \in \AF_D. \]
\begin{definition}
    S-series over decorated aromatic forests are defined as $S := F_f \circ \delta_\sigma$, that~is,
    \[ S(a) = \sum_{\pi \in AF_D} \frac{a(\pi)}{\sigma(\pi)} F_f (\pi). \]
\end{definition}

\begin{theorem}\textbf{(composition law)} $\;$ \label{thm:composition_law}
    Let $S(a)$ and $S(b)$ be two S-series and let $\phi$ be a test function. Then,
    \[ S(a)[S(b)[\phi]] = S(a * b)[\phi], \]
    where $a * b = m_\mathbb{R} \circ (a \otimes b) \circ \Delta_{CK}$.
\end{theorem}
\begin{proof}
    We use identity $\delta_\sigma \circ * = \gl \circ (\delta_\sigma \otimes \delta_\sigma)$, which follows from Proposition \ref{prop:Asigma_general}, and the definition of the Grossman-Larson product.
\end{proof}

We are interested in studying how the structures that were originally defined on deterministic differential equations are affected when considered in a stochastic context. Similarly to the classical Runge-Kutta methods which can be written as B-series, sRK methods (Definition \ref{RK_form}) can be written using B-series over grafted trees \cite{Laurent_2019}. Composing sRK methods between themselves or with a test function gives rise to S-series over grafted forests. For example, let $\Psi_h$ be an sRK method and $I : \mathbb{R}^d \to \mathbb{R}^l$, then,
\[ \Psi_h(y_0) = y_0 + \sum_{\tau \in AT_g} \frac{a(\tau)}{\sigma(\tau)} F_f (\tau)(y_0), \quad \text{and} \quad I \circ \Psi_h = I + \sum_{\pi \in AF_g}\frac{a(\pi)}{\sigma(\pi)}F_f(\pi)[I], \]
with the functional $a : AF_g \to \mathbb{R}$ defined in Proposition \ref{prop:RK_a}.
Differential operators $F_f(\pi)$ corresponding to grafted forests are defined by extending Definition \ref{def:F_f_classic} with $F_f (\times) = \sqrt{2h} \xi$ where $\xi \sim \mathcal{N}(0, I_d)$. For example, 
\[ F_f(\forest{b[b[x],b],x,x}) = h^{4.5} 2 \sqrt{2} \sum_{i,j,k=1}^d \xi^i (\partial_i f^j) f^k (\partial_{jk} f) \xi \xi. \] 
The map $F_f$ defined this way over $\AF_g$ is an algebra morphism, therefore, the composition law for S-series over grafted forests is given by Theorem \ref{thm:composition_law}.

The functional $a : \AF_g \to \mathbb{R}$ corresponding to an sRK method with coefficients $b_i, d_i, a_{ij}$ with $i, j = 1, \dots, s$ is defined analogously to the classical case, that is, as a sum over sRK coefficients with $b_i$ corresponding to the black roots, $a_{ij}$ corresponding to black vertices which are not roots, and $d_i$ corresponding to the grafted vertices. Grafted roots coorespond to $1$. The coefficient is $0$ on all grafted forests that contain an aroma.

\begin{proposition}\label{prop:RK_a}
    Let $S(a)$ be the S-series of a stochastic Runge-Kutta method with coefficients $b_i, a_{ij}, d_i$ for $i, j \in [s]$, then the map $a : AF_g \to \mathbb{R}$ is $0$ on $AF_g \setminus F_g$ and 
    \[ a(\pi) := \sum_{\alpha : V(\pi) \to [s]} \prod_{\substack{v \in V_\bullet (\pi) \\ s(v) = \emptyset}} b_{\alpha(v)} \cdot \prod_{\substack{v \in V_\bullet (\pi) \\ u \in s(v)}} a_{\alpha(u), \alpha(v)} \cdot \prod_{\substack{v \in V_\times (\pi) \\ u \in s(v)}} d_{\alpha(u)} , \quad \text{for } \pi \in F_g, \]
    where $F_g$ is the set of grafted forests without aromas, $V_c(\pi)$ are the vertices of $\pi$ of color $c$, and $s(v)$ are the successors of $v$ in $\pi$.
\end{proposition}

We recall that all edges are directed towards the roots of the corresponding connected components. Some values of $a : AF_g \to \mathbb{R}$ corresponding to an sRK method with coefficients $b_i, a_{ij}, d_i$ for $i,j \in [s]$ are 
\[ a(\forest{x}) = 1, \quad a(\forest{b}) = \sum_{i=1}^s b_i, \quad a(\forest{b[x]}) = \sum_{i=1}^s b_i d_i \quad a(\forest{b[b]}) = \sum_{i,j=1}^s b_i a_{ij}. \]
Some values of $a$ for bigger forests are:
\[ a(\forest{b[x,b],x}) = \sum_{i,j=1}^s b_i d_i a_{ij}, \quad a(\forest{b[b,x,b[b,x]]}) = \sum_{i,j,k,l=1}^s b_i a_{ij} d_i a_{ik} a_{kl} d_k, \quad a(\forest{b[x],b[b]}) = \sum_{i,j,k = 1}^s b_i d_i b_j a_{jk}.  \]
The proof of Proposition \ref{prop:RK_a} is a straightforward extension of the theory presented in \cite{GeomIntBook} in Chapter III.1.1. 
The functional $a : AF_g \to \mathbb{R}$ corresponding to an S-series for sRK method is $0$ on all grafted forests that contain aromas due to the fact that the Taylor expansion of $I(\Psi_h(y_0))$ around $y_0$ does not produce differentials that would correspond to aromas.

\subsection{From grafted to exotic forests}

Since we are interested in studying the order conditions with respect to the invariant measure, we consider the expansion of $\E[\phi(y_0 + B(a)(y_0))]$ which can be written as $\E[S(a)[\phi]]$ where $S(a)$ is an S-series over grafted forests. Let us consider how the expectation acts on the differential operators corresponding to grafted forests~\cite{Laurent_2019}.
From the definition of $F_f$ on grafted forests, it follows that the expectation depends only on the grafted vertices, i.e. on terms of the form $\mathbb{E} (\xi_{i_1} \cdots \xi_{i_m})$. 
We know that the expectation is $0$ for $m$ odd, thus, we consider $\mathbb{E} (\xi_{i_1} \cdots \xi_{i_{2n}})$. We know that $\mathbb{E}(\xi_i \xi_j) = \mathbb{E} (\xi_i) \mathbb{E} (\xi_j) $ if $i \neq j$, therefore, the indices must have even multiplicities. For example,
\begin{align*}
\E [F_f(\forest{b[x,b[x]]})] = &\E [ 2 h^3 \sum_{i,j,k=1}^d \xi^i \partial_i f^j \xi^k \partial_{k,j} f] = 2 h^3 \sum_{i,j,k=1}^d \E [\xi^i \xi^k] \partial_i f^j \partial_{k,j} f  = \\ 
& 2 h^3 \sum_{i,j=1}^d \E [\xi^i \xi^i] \partial_i f^j \partial_{i,j} f = 2 h^3 \sum_{i,j=1}^d \partial_i f^j \partial_{i,j} f.
\end{align*}
We notice that the expectation forces $i$ and $k$, which are indices corresponding to the grafted vertices, to be equal. This creates a pairing between the grafted vertices, this pairing is named \textit{liana} and the resulting vector field corresponds to the vector field of an exotic forest. Therefore, the expectation of an S-series over grafted forests is an exotic S-series, i.e. S-series over exotic forest.

Let $n \in \N$ be a natural number and let $e_1, \dots, e_n$ denote the standard basis of $\mathbb{R}^n$. Let $g : \mathbb{R}^d \to \mathbb{R}^d$ be a vector field, then $g(x) = \sum_{i=1}^\infty g^i (x) e_i$ with $g^i = 0$ for all $i > d$. Let $g(x) = \sqrt{2h} \mathbf{1} \in \mathbb{R}^d$ where $\mathbf{1}$ is a vector of ones. 
The map $F_f$ is defined over exotic forests by extending Definition \ref{def:F_f_classic} with $F_f(k) = g^k e_k$ for $k \in \mathbb{N}$ using the construction (\ref{eq:exotic_def}) of exotic forests as a subspace of $\widetilde{\AF}_{\bullet \mathbb{N}}$.
We use the example we have already seen to illustrate the definition:
\[ F_f(\forest{b[1,b[1]]}) = 2 h^3 \sum_{i,j,l=1}^d \partial_i f^j \partial_{i,j} f^l \partial_l. \]
The map $F_f$ defined this way over $\EF$ is an algebra morphism, therefore, the composition law for exotic S-series is given by Theorem \ref{thm:composition_law}.
Theorem \ref{thm:Laurent} is proved in \cite{Laurent_2019} and follows from the Isserlis Theorem \cite{Isserlis_18}.
\begin{theorem}\label{thm:Laurent}\cite{Laurent_2019}
    Let $(\pi, \alpha_g) \in AF_g$ be a grafted forest with an even number of grafted vertices. Then, the expectation of $F_f(\pi, \alpha_g)$ is given by
    \[ \E [F_f(\pi, \alpha_g)] = \sum_{\alpha_e} F_f(\pi, \alpha_e), \]
    where the sum is over all decorations $\alpha_e$ with $\alpha_e^{-1} (\bullet) = \alpha_g^{-1} (\bullet)$ where $\alpha_e$ is defined by (\ref{eq:exotic_def}).
\end{theorem}
Using Theorem \ref{thm:Laurent}, we can define $\E : \AF_g \to \EF$ with $\E \circ F_f = F_f \circ \E$. For example,
\[ \mathbb{E} \big[ \forest{b[x,x,b[x,x]]} \big] = \forest{b[1,1,b[2,2]]} + 2 \forest{b[1,2,b[1,2]]}. \]

Let $\Phi : \{\bullet\} \sqcup \mathbb{N} \to \{\bullet,\times\}$ be a map defined as $\Phi (\bullet) = \bullet$ and $\Phi(k) = \times$ for $k \in \mathbb{N}$ that induces a map on decorations $\Phi(\alpha_e) = \Phi \circ \alpha_e$ and on exotic forests $\Phi(\pi, \alpha_e) = (\pi, \Phi \circ \alpha_e)$.
For example, $\Phi(\forest{b[1,1]}) = \forest{b[x,x]}$, $\Phi(\forest{b[1,2],b[2],1}) = \forest{b[x,x],b[x],x}$.

\begin{corollary}\label{corr:ES_expS}
    The expectation of an S-series over grafted forests $S(a)$ is an exotic S-series $ES(a \circ \Phi)$, that is,
    \[ \mathbb{E} \left[ S(a) \right] = ES(a \circ \Phi). \]
\end{corollary}
\begin{proof}
    Recall that $p(\pi, \alpha_e, \alpha_g)$ (Section \ref{sec:rel_two_decorations}) is the number of ways to pair grafted vertices of $(\pi, \alpha_g)$ to obtain an exotic forest isomorphic to $(\pi, \alpha_e)$. Therefore, using Proposition \ref{prop:p_generalized}, we have,
    \[ \E[(\pi, \alpha_g)] = \sum_{\alpha_e \in \Phi^{-1}(\alpha_g)} p(\pi, \alpha_e, \alpha_g) (\pi, \alpha_e) = \sum_{\alpha_e \in \Phi^{-1}(\alpha_g)} \frac{\sigma(\pi, \alpha_g)}{\sigma(\pi, \alpha_e)} (\pi, \alpha_e), \]
    where the sum is over all $\alpha_e$ such that $\alpha_e^{-1} (\bullet) = \alpha_e^{-1} (\bullet)$. We use this identity to see that $\E$ commutes with $\delta_\sigma$. Since $\E$ commutes with both $F_f$ and $\delta_\sigma$, the statement is proved.
\end{proof}

\begin{remark}
We note that, in general, the expectation of a composition of S-series is not equal to the composition of the corresponding exotic S-series, i.e.
\[ \E \left[ S(a)[S(b)] \right] \neq ES(a) [ES(b)]. \]
    However, the equality holds if we use a splitting method that splits the noise denoted by the grafted vertices $\times$ into noises $\times_1$ and $\times_2$ that are independent. This assures that $\E \left[ S(a)[S(b)] \right]$ does not contain pairings of grafted vertices of colors $\times_1$ and $\times_2$, and makes $S(a)$ and $S(b)$ independent random variables. 
\end{remark}

\section{Order conditions for invariant measure sampling of ergodic SDEs}\label{sec:Coef_mult}

In this section, we generate order conditions building on the results and ideas from \cite{Laurent_2019}. We define a theoretical algorithm that generates systematically the order conditions for the invariant measure sampling of ergodic SDEs and prove an algebraic property of the generated order conditions. This allows us to reduce the number of order conditions, for example, it allows us to reduce the number of order 3 conditions from $13$ to $9$. 

We consider a truncated exotic S-series and apply transformations to the corresponding exotic forests. The result of the algorithm is a truncated exotic S-series $ES(\omega)$. We prove that $\omega$ is a character of $(\EF, \cdot)$, therefore, the order $p$ conditions with respect to the invariant measure are obtained by requiring $\omega(\tau) = 0$ for all $\tau \in ET$ with $|\tau| < p$.

\subsection{Numerical analysis framework} \label{sec:num_framework}

We recall that we focus on the overdamped Langevin equation (\ref{eq:problem}), 
\[ dX(t) = f(X(t))dt + \sqrt{2} dW(t), \quad X(t) \in \mathbb{R}^d, \quad f = - \nabla V, \]
where $V : \mathbb{R}^d \to \mathbb{R}$ is a smooth and globally Lipschitz potential and there exist $C_1 > 0$ and $C_2$ such that for all $x \in \mathbb{R}^d$, $V(x) \geq C_1 x^Tx - C_2$.
Such potential guarantees the problem to be ergodic~\cite{Mattingly_2002} with the density of the unique invariant measure being $\rho_\infty = Z \exp(-V)$ where $Z$ is such that $\int_{\mathbb{R}^d} {\rho_\infty (x)dx} = 1$.

\begin{assumption}
	The integrator $X_1 = \Psi_h(X_0, f, \xi)$ has bounded moments of any order along time, i.e., for all integer $k \geq 0$, 
$$ \sup_{n \geq 0} {\mathbb{E} [ | X_n |^{2k} ] } < \infty \quad \forall k \geq 0 $$
\end{assumption}

\begin{assumption}\label{assumption:Taylor} 
	The integrator $X_1 = \Psi_h(X_0, f, \xi)$ has a weak Taylor expansion of the form
$$ \mathbb{E} [ \phi(X_1) | X_0 = x ] = \phi(x) + h \mathcal{A}_1 \phi(x) + h^2 \mathcal{A}_2 \phi(x) + \cdots $$
    for all $\phi \in C^\infty_P (\mathbb{R}^d, \mathbb{R})$, where $\mathcal{A}_i, \ i = 1,2,\dots$, are linear differential operators. For more details see \cite{Talay_Tubaro}. We assume that $\mathcal{A}_1 = \mathcal{L}$ where $\mathcal{L}$ is the generator discussed in Section $\ref{sec:num_intro}$, that is, the integrator has at least weak order $1$.
\end{assumption}

\begin{theorem}\label{main_order_cond_thm} \cite{Abdulle_Vilmart_Zygalakis}
	Take an ergodic integrator $X_1 = \Psi_h (X_0, f, \xi)$.
        Assume the Assumptions 1, 2 to be true. If, for all $\phi \in C_P^\infty(\mathbb{R}^d, \mathbb{R})$, we have
    $$ \int_{\mathbb{R}^d} \mathcal{A}_j \phi (x) \rho_\infty(x) dx = 0, \quad j = 2, \dots, p, $$
	then, the integrator has order $p$ with respect to the invariant measure.
\end{theorem}

\subsection{Transformation of exotic forests}

We consider numerical integrators that can be expanded using B-series over grafted trees, for example, sRK methods. We recall that for such integrators the differential operators $\mathcal{A}_j$ from Theorem \ref{main_order_cond_thm} have the form
\[ \mathcal{A}_j = \sum_{\pi \in EF_j} \frac{a(\pi)}{\sigma(\pi)} F_f (\pi)[\cdot], \]
where $EF_j$ is the subset of exotic forests of size $j$. Thus, Theorem \ref{main_order_cond_thm} states that a numerical method $y_0 + B(a)(y_0) $ is order $p$ with respect to the invariant measure, if
\[ \int_{\mathbb{R}^d} ES_{<p} (a)[\phi] \rho_\infty dx = (I \circ \delta_{\sigma, <p})(a) = 0, \]
where $I(\pi) = \int_{\mathbb{R}^d} F_f (\pi)[\phi] \rho_\infty dx$ and $\delta_{\sigma, <p} (a)$ is the sum over all exotic forests up to size $p$ with coefficients given by $a : EF \to \mathbb{R}$ normalized by $\sigma$. We obtain order conditions with respect to the invariant measure by modifying the differential operators that make up $ES_{<p}(a)$ in a way that does not change the value of the integral. This translates into two transformations applied to the corresponding exotic forests:
\begin{enumerate}
    \item Edge-liana inversion (ELI), which moves the liana down the tree along an edge,
        \[ \forest{b_A_180[b_B_180[1]],b_C[1]} \quad \xrightarrow{ELI} \quad \forest{b_A_180[1],b_C[b_B[1]]}, \]
        where we note that ELI uses the fact that the exotic forests here are used to denote differential operators, which means that it assumes there is an "invisible" edge starting at the roots,
    \item Integration by parts (IBP), which takes a grafted root, connects it to all other vertices with coefficient $-1$, and adds a term with coefficient $-2$ in which the grafted root is removed and the paired grafted vertex is colored black, for example,
        \[ \forest{1_A_180,b[b,b[1]]} \quad \xrightarrow{IBP} \quad - \forest{b[1_A_180,b,b[1]]} - \forest{b[b[1_A_180],b[1]]} - \forest{b[b,b[1_A_180,1]]} - 2 \forest{b[b,b[b]]}. \]
\end{enumerate}
 More details on ELI and IBP can be found in \cite{Laurent_2019} in Section 4.2 and Proposition 4.7.
Proposition \ref{prop:sim_integral} allows us to use ELI and IBP to obtain order conditions. Proposition is proven for ELI using the fact that the vector field $f$ is the gradient of a potential, that is, $f = - \nabla V$, and for IBP using the integration by parts process on the integral.

\begin{proposition} \label{prop:sim_integral} \cite{Laurent_2019}
    Let $\pi_1, \hat{\pi}_1, \pi_2, \hat{\pi}_2 \in \EF$ such that $\pi_1 \xrightarrow{ELI} \hat{\pi}_1$ and $\pi_2 \xrightarrow{IBP} \hat{\pi}_2$, then 
    \[ I(\pi_1) = I(\hat{\pi}_1) \quad \text{and} \quad I(\pi_2) = I(\hat{\pi}_2), \]
where $\phi \in C_P^\infty$ from the definition of $I$ is a test function.
\end{proposition}
For example, let $\forest{b[1],b[1]} \xrightarrow{ELI} \forest{1,b[b[1]]}$, then, $I(\forest{b[1],b[1]}) = I(\forest{1,b[b[1]]})$ due to the following computation where the property $\partial_i f^j = \partial_j f^i$ is used,
\begin{align*}
    I(\forest{b[1],b[1]}) &= 2h \sum_{i,j,k=1}^d \int_{\mathbb{R}^d} (\partial_i f^j) (\partial_i f^k) (\partial_{j,k} \phi) \rho_\infty dx \\
    &= 2h \sum_{i,j,k=1}^d \int_{\mathbb{R}^d} (\partial_j f^i) (\partial_i f^k) (\partial_{j,k} \phi) \rho_\infty dx = I(\forest{1,b[b[1]]}).
\end{align*}
Analogously, let $\forest{1,b[b[1]]} \xrightarrow{IBP} - \forest{b[b[1,1]]} - \forest{b[1,b[1]]} - 2 \forest{b[b[b]]}$, then, we perform the following computation,
\begin{align*}
    I(\forest{1,b[b[1]]}) &= 2h \sum_{i,j,k=1}^d \int_{\mathbb{R}^d} (\partial_i f^j) (\partial_j f^k) (\partial_{i,k} \phi) \rho_\infty dx \\
    &= 2h \sum_{i,j,k=1}^d \big[ - \int_{\mathbb{R}^d} (\partial_{i,i} f^j) (\partial_j f^k) (\partial_k \phi) \rho_\infty dx 
    - \int_{\mathbb{R}^d} (\partial_i f^j) (\partial_{i,j} f^k) (\partial_k \phi) \rho_\infty dx \\
    &\quad - \int_{\mathbb{R}^d} (\partial_i f^j) (\partial_j f^k) (\partial_k \phi) (\partial_i \rho_\infty) dx \big]
    = I(- \forest{b[b[1,1]]} - \forest{b[1,b[1]]} - 2 \forest{b[b[b]]}).
\end{align*}
where we used integration by parts of the integral and the identity $\partial_i \rho_\infty = f^i \rho_\infty$.
We define the term \textit{connecting liana} and build an algorithm by composing ELI and IBP such that the exotic forests obtained by the algorithm have no connecting lianas. 
\begin{definition}
    Let a connecting liana be a liana $\alpha_e^{-1} (k) = \{v_1, v_2\}$ in $\pi$ for some $k \in \mathbb{N}$ such that $v_1$ and $v_2$ are in different connected components of $\pi$.
\end{definition}
\begin{example}\label{ex:alg_action}
Connecting lianas are labeled in exotic forests below,
\[ \forest{b[1_A_90],b[1_A_90]}, \quad \forest{b[1,1,b[2_B_90]],b[2_B_90]}, \quad \forest{1_C_90,b[b[b,1_C_90]]}, \]
and the compositions of ELI and IBP that get rid of the connecting lianas are
\[ \forest{b[1],b[1]} \quad \xrightarrow{ELI} \quad \forest{1,b[b[1]]} \quad \xrightarrow{IBP} \quad - \forest{b[b[1,1]]} - \forest{b[1,b[1]]} - 2 \forest{b[b[b]]}, \]
\[ \forest{b[1,1,b[2]],b[2]} \quad \xrightarrow{ELI} \quad \forest{b[1,1,2],b[b[2]]} \quad \xrightarrow{ELI} \quad \forest{2,b[b[b[1,1,2]]]} \quad \xrightarrow{IBP} \quad - \forest{b[2,b[b[1,1,2]]]} - \forest{b[b[2,b[1,1,2]]]} - \forest{b[b[b[2,1,1,2]]]} - 2 \forest{b[b[b[1,1,b]]]}, \]
\[ \forest{1,b[b[b,1]]} \quad \xrightarrow{IBP} \quad - \forest{b[1,b[b,1]]} - \forest{b[b[1,b,1]]} - \forest{b[b[b[1],1]]} - 2 \forest{b[b[b,b]]}. \]
\end{example}

The compositions of ELI and IBP listed in Example \ref{ex:alg_action} are called transformation chains and denoted by $\pi \to \hat{\pi}$ where $\pi \in EF$ and $\hat{\pi} \in \widehat{\EF}$, where $\widehat{\EF}$ is the vector space of exotic forests without connecting lianas.

\subsection{Labeled Transformation Chains}\label{sec:labeled_transformation_chains}

To simplify the analysis of the algorithm that we introduce, let us consider the space of labeled exotic forests denoted by $\EF_L$. We use the labeling to split the transformation chains into labeled transformation chains (LTCs) that have labeled exotic forests as terms. This means that the IBP transformation applied to $\pi$ is split into $IBP_v$ transformations for $v \in V(\pi)$ and $IBP_\bullet$ for the term where the grafted root is removed and the remaining grafted vertex becomes black. For example, the transformation chain
\[ \forest{b[1],b[1]} \quad \xrightarrow{ELI} \quad \forest{1,b[b[1]]} \quad \xrightarrow{IBP} \quad - \forest{b[b[1,1]]} - \forest{b[1,b[1]]} - 2 \forest{b[b[b]]} \]
is split into the following labeled transformation chains:
\[ \forest{b[1],b[1]} \quad \xrightarrow{ELI} \quad \forest{1,b[b[1]]} \quad \xrightarrow{IBP_1} \quad \forest{b[b[1,1]]}, \]
\[ \forest{b[1],b[1]} \quad \xrightarrow{ELI} \quad \forest{1,b[b[1]]} \quad \xrightarrow{IBP_2} \quad \forest{b[1,b[1]]}, \]
\[ \forest{b[1],b[1]} \quad \xrightarrow{ELI} \quad \forest{1,b[b[1]]} \quad \xrightarrow{IBP_\bullet} \quad \forest{b[b[b]]}, \]
where we exclude the coefficients from the LTC and handle them separately. We note that ELI is not affected.

Let us denote by $\Psi : \EF \to \EF_L$ any injection of $\EF$ into $\EF_L$ which labels the vertices of exotic forests according to some rules. Let $\Phi : \EF_L \to \EF$ be the linear map that forgets the labeling. We note that $\Phi \circ \Psi = \text{id}$ and $\Psi \circ \Phi$ is a relabeling. Let $A : \EF \to \widehat{\EF}$ be a linear map defined as
\[ A := \Phi \circ A_L \circ \Psi, \quad \text{where } A_L (\pi) := \sum_{\pi \to \hat{\pi}} C(\pi \to \hat{\pi}) \hat{\pi}, \]
where the sum is taken over all labeled transformation chains (LTCs) starting at $\pi$ which are generated recursively by Algorithm 1 and the coefficient $C(\pi \to \hat{\pi})$ is defined as
\[C(\pi \to \hat{\pi}) := (-1)^{|\pi \to \hat{\pi}|_{IBP_v}} (-2)^{|\pi \to \hat{\pi}|_{IBP_\bullet}},\]
where $|\pi \to \hat{\pi}|_{IBP_v}$ is the number of $IBP_v$ transformations for $v \in V(\tilde{\pi})$ for any intermediate $\tilde{\pi}$ and $|\pi \to \hat{\pi}|_{IBP_\bullet}$ is the number of $IBP_\bullet$ transformations.

Let us assume a total order on the vertices of any labeled exotic forest. We require the total order to respect the concatenation product, that is,
\[ v_1 \leq v_2 \text{ in } \pi_1 \cdot \pi_2 \quad \text{if} \quad v_1 \leq v_2 \text{ in } \pi_1, \quad \text{for } v_1, v_2 \in V(\pi_1).  \]
Such order can be obtained by extending the results from \cite{TygliyanFilippov}. Let a minimal connecting liana be the connecting liana $\{ v_1, v_2 \}$ such that $v_1$ has the shortest path to the root. If there are multiple such lianas, choose the liana with smallest $v_1$ according to the total order of vertices. If there are multiple lianas with equal $v_1$, choose the liana with the smallest $v_2$.

\vspace{.5cm}
\textbf{Algorithm 1:} Generate the set of all LTCs that start with $\pi_1 \in EF_L$\\
\rule{\textwidth}{.1pt}
\textbf{Input:} $\pi_1 \in EF_L$ \\
\textbf{Output:} The set of LTCs $\{ \pi_1 \to \pi_n \}$ with $\pi_n \in \widehat{EF}_L$.
\begin{enumerate}
    \item[\texttt{Step 1: }] If $\pi_1$ has no connecting lianas, then return the singleton set $\{ \pi_1 \to \pi_1 \}$. Else, let $l = \{v_1, v_2\}$ be the minimal connecting liana of $\pi_1$.
    \item[\texttt{Step 2: }] If the grafted vertex $v_1$ is a root, then let $\{\pi_{2}^{(i)}\}_{i=1}^N$ with $N = |V(\pi_1)| + 1$ be the set of forests obtained by applying $IBP_v$ and $IBP_\bullet$ to $\pi_1$ with respect to $l$. 
    \item[\texttt{Step 3: }] If $v_1$ is not a root, then let $\{ \pi_{2}\}$ be the singleton set containing the forest obtained by applying ELI moving $l$ towards the root. 
    \item[\texttt{Step 4: }] For each $\pi_{2} \in \{\pi_{2}^{(i)}\}_{i=1}^N$, apply \textbf{Algorithm 1} to $\pi_2$. Merge all the resulting sets $\{ \pi_2^{(i)} \to \pi_n \}$ of LTCs and prepend $\pi_1$ to each LTC. Return the resulting set $\{\pi_1 \to \pi_n\} = \{ \pi_1 \to \pi_2^{(i)} \to \pi_n \}$.
\end{enumerate}
\rule{\textwidth}{.1pt}
\vspace{.5cm}

\begin{proposition}
	Algorithm 1 ends in a finite number of steps.
\end{proposition}
\begin{proof} 
The algorithm is guaranteed to end because every application of the IBP decreases the number of roots, which means that IBP can be applied only a finite number of times. The application of ELI does not change the number of roots. ELI moves a liana towards the root which can be done a finite number of times. Note that the minimal connecting liana will stay minimal after the application of ELI.
\end{proof}

Let $\langle \cdot, \cdot \rangle $ be the orthonormal inner product, that is, for $\pi_1, \pi_2 \in EF$, we have
\[ \langle \pi_1, \pi_2 \rangle := \begin{cases} 1 &, \text{if } \pi_1 = \pi_2, \\ 0 &, \text{otherwise.} \end{cases} \]
Let $\langle \cdot, \cdot \rangle_\sigma$ be the renormalized inner product, that is, $\langle \pi_1, \pi_2 \rangle_\sigma := \sigma(\pi_1) \langle \pi_1, \pi_2 \rangle$. We note that both inner products are equal on the space of labeled exotic forests $\EF_L$.
Due to Algorithm 1, the maps $A_L$ and $A$ are well-defined and we are ready to obtain the order conditions with respect to the invariant measure. The order conditions are denoted by $\omega(\pi) = 0$ with $\pi \in \widehat{EF}$ where
\begin{equation}\label{eq:omega}
    \omega (\pi) := \langle (A \circ \delta_\sigma) (a), \pi \rangle_\sigma.
\end{equation}
Due to Theorem \ref{main_order_cond_thm} and Proposition \ref{prop:sim_integral}, the conditions $\omega (\pi) = 0$ for all $\pi \in \widehat{EF}, |\pi| < p$, imply the order $p$ with respect to the invariant measure, since
\[ \int_{\mathbb{R}^d} ES_{<p}(a)[\phi]\rho_\infty dx = (I \circ \delta_{\sigma, <p}) (a) = (I \circ A \circ \delta_{\sigma, <p}) (a) = 0. \]
The values of $\omega$ for all exotic forests up to size $3$ are listed in Table \ref{table:append_ord_cond} of the Appendix.

\subsection{Multiplicative property of order conditions}

Let $\Delta_\sigma : \EF \to \EF \otimes \EF$ denote the dual of the concatenation product with respect to the inner product $\langle \cdot, \cdot \rangle_\sigma$. The explicit formula for $\Delta_\sigma$ is the following
\[ \Delta_\sigma (\frac{\pi}{\sigma(\pi)}) = \sum_{\pi_1 \cdot \pi_2 = \pi} \frac{\pi_1}{\sigma(\pi_1)} \otimes \frac{\pi_2}{\sigma(\pi_2)}. \]
We can see that this formula is true, since,
\[ \langle \pi_1 \cdot \pi_2, \pi \rangle_\sigma = \langle \pi_1 \otimes \pi_2, \Delta_\sigma (\pi) \rangle_\sigma = \sigma(\pi), \quad \text{if } \pi_1 \cdot \pi_2 = \pi. \]
Let us also consider the dual of the concatenation product on the space $\EF_L$ of labeled exotic forests, $\Delta : \EF_L \to \EF_L \otimes \EF_L$. We prove Lemma \ref{lemma:LTC_split_sets} as an intermediate result.

\begin{lemma}\label{lemma:LTC_split_sets}
Let $\pi, \hat{\pi} \in EF_L$ be labeled exotic forests. We define the sets $S_1$ and $S_2$ as
\begin{align*}
S_1 &:= \{ (\pi \to \hat{\pi}, (\hat{\pi}_1, \hat{\pi}_2)) \; : \; \hat{\pi}_1 \cdot \hat{\pi}_2 = \hat{\pi} \}, \\
S_2 &:= \{ ((\pi_1, \pi_2), \pi_1 \to \hat{\pi}_1, \pi_2 \to \hat{\pi}_2) \; : \; \pi_1 \cdot \pi_2 = \pi \},
\end{align*}
then, $S_1 \cong S_2$.
\end{lemma}
\begin{proof}
Let us take a tuple $(\pi \to \hat{\pi}, (\hat{\pi}_1, \hat{\pi}_2)) \in S_1$. It contains an LTC $\pi \to \hat{\pi}$ and a splitting of $\hat{\pi}$ into $\hat{\pi}_1$ and $\hat{\pi}_2$. Since LTC keeps the labels of vertices when it acts on them, we can split $\pi$ into $\pi_1$ and $\pi_2$ by following the labeling of $\hat{\pi}_1$ and $\hat{\pi}_2$. This also gives us a splitting of the LTC $\pi \to \hat{\pi}$ into $\pi_1 \to \hat{\pi}_1$ and $\pi_2 \to \hat{\pi}_2$. That is, we get a tuple $((\pi_1, \pi_2), \pi_1 \to \hat{\pi}_1, \pi_2 \to \hat{\pi}_2)$ which is an element of $S_2$. 

Let us take a tuple $((\pi_1, \pi_2), \pi_1 \to \hat{\pi}_1, \pi_2 \to \hat{\pi}_2) \in S_2$ that contains a splitting of $\pi$ into $\pi_1$ and $\pi_2$, and two LTCs $\pi_1 \to \hat{\pi}_1$ and $\pi_2 \to \hat{\pi}_2$. We can combine $\pi_1 \to \hat{\pi}_1$ and $\pi_2 \to \hat{\pi}_2$ by concatenating all intermediate labeled exotic forests into one LTC $\pi \to \hat{\pi}$. This is possible since the total order of vertices respects the concatenation product and the two LTC have distinct labels because $\pi_1$ and $\pi_2$ are a splitting of one exotic forest. By combining the two LTCs, we also get an exotic forest $\hat{\pi}$ that has $\hat{\pi}_1$ and $\hat{\pi}_2$ as splitting.

This finishes the proof.
\end{proof}

We prove that the maps $\Phi : (\EF_L, \Delta) \to (\EF, \Delta_\sigma)$, $A_L : (\EF_L, \Delta) \to (\widehat{\EF}_L, \Delta)$, and $A : (\EF, \Delta_\sigma) \to (\widehat{\EF}, \Delta_\sigma)$ are coalgebra morphisms. We recall that the maps $A$ and $A_L$ are defined in Section \ref{sec:labeled_transformation_chains} and $\Phi$ is the map that forgets the labels of the exotic forests.

\begin{proposition}\label{prop:Alg_LD}
    The following identities are true:
    \begin{enumerate}
        \item $\Delta_\sigma \circ \Phi = (\Phi \otimes \Phi) \circ \Delta$,
        \item $\Delta \circ A_L = (A_L \otimes A_L) \circ \Delta$,
        \item $\Delta_\sigma \circ A = (A \otimes A) \circ \Delta_\sigma$.
    \end{enumerate}
\end{proposition}
\begin{proof} We first prove identities $(1)$ and $(2)$ and use them to prove identity $(3)$.

    Step 1) Let us write the labeling of an exotic forests explicitly as $\alpha$ and take $(\pi, \alpha) \in EF_L$, then, we have
        \[ ((\Phi \otimes \Phi) \circ \Delta)(\pi, \alpha) = \sum_{\substack{\pi_1 \cdot \pi_2 = \pi \\ \alpha_1 \sqcup \alpha_2 = \alpha}} \pi_1 \otimes \pi_2 \]
        We note the splittings of $\alpha$ as $\alpha_1 \sqcup \alpha_2$ with $\alpha_i$ being the decoration of $\pi_i$ for $i = 1,2$. Every splitting of $\alpha$ corresponds to a decorated exotic forest $(\pi, \beta)$ with $\beta : V(\pi) \to \{1,2\}$ such that $\beta^{-1} (1) = \pi_1$ and $\beta^{-1}(2) = \pi_2$, therefore, 
        \[ |\{ (\alpha_1, \alpha_2) \; : \; \alpha_1 \sqcup \alpha_2 = \alpha \}| = |\{ \beta : V(\pi) \to \{1,2\} \; : \; \beta^{-1}(i) = \pi_i \}| = |B|. \]
        We see that $|B| = p(\pi, \beta, \emptyset)$ and, from Proposition \ref{prop:p_generalized}, we know that 
        \[p(\pi, \beta, \emptyset) = \frac{\sigma(\pi)}{\sigma(\pi, \beta)} = \frac{\sigma(\pi)}{\sigma(\pi_1)\sigma(\pi_2)}.\]
        This implies that, using the formula for $\Delta_\sigma$, we have
        \begin{align*}
            ((\Phi \otimes \Phi) \circ \Delta)(\pi, \alpha) &= \sigma(\pi) \sum_{\pi_1 \cdot \pi_2 = \pi} \frac{\pi_1}{\sigma(\pi_1)} \otimes \frac{\pi_2}{\sigma(\pi_2)} \\
            &= \sigma(\pi) \Delta_\sigma (\frac{\pi}{\sigma(\pi)}) = (\Delta_\sigma \circ \Phi) (\pi, \alpha).
        \end{align*}
        This proves identity $(1)$.

    Step 2) We follow the definition of $A_L$ and use the property $C(\pi \to \hat{\pi}) = C(\pi_1 \to \hat{\pi}_1) C(\pi_2 \to \hat{\pi}_2)$ where the LTC $\pi \to \hat{\pi}$ splits into the LTCs $\pi_1 \to \hat{\pi}_1$ and $\pi_2 \to \hat{\pi}_2$. We start by using the definitions of $A_L$ and $\Delta$ and obtain:
        \[ \Delta \circ A_L = \sum_{\substack{\pi \to \hat{\pi} \\ \hat{\pi}_1 \cdot \hat{\pi}_2 = \hat{\pi}}} C(\pi \to \hat{\pi}) (\hat{\pi}_1 \otimes \hat{\pi}_2). \]
        Then, we use Lemma \ref{lemma:LTC_split_sets} and group the terms to get
        \[ \Delta \circ A_L = \sum_{\pi_1 \cdot \pi_2 = \pi}  \Big( \sum_{\pi_1 \to \hat{\pi}_1} C(\pi_1 \to \hat{\pi}_1) \hat{\pi}_1 \Big) \otimes \Big( \sum_{\pi_2 \to \hat{\pi}_2} C(\pi_2 \to \hat{\pi}_2) \hat{\pi}_2 \Big). \]
        We use the definitions of $A_L$ and $\Delta$ to conclude the proof of identity $(2)$.

    Step 3) We use the definition of $A$ and identites $(1)$ and $(2)$ to show that
        \[ \Delta_\sigma \circ A = (\Phi \otimes \Phi) \circ (A_L \otimes A_L) \circ \Delta \circ \Psi. \]
        We note that the definition of $A$ accepts any injection $\Psi$, therefore, we can insert a relabeling $\Psi \circ \Phi$ to obtain
        \[ \Delta_\sigma \circ A = (\Phi \circ \Phi) \circ (A_L \otimes A_L) \circ (\Psi \otimes \Psi) \circ (\Phi \otimes \Phi) \circ \Delta \circ \Psi, \]
        which proves identity $(3)$ using the definition of $A$, identity $(1)$, and the property $\Phi \circ \Psi = \text{id}$.
\end{proof}

\begin{theorem}\label{thm:Coef_mult}
    Let us apply Algorithm 1 to an exotic S-series $ES(a)$ with $a$ being a character of $(\EF, \cdot)$, then, the map $\omega$ defined as (\ref{eq:omega}) is a character of $(\EF, \cdot)$, that is,
    \begin{equation}\label{eq:Coef_mult}
        \omega(\pi_1 \cdot \pi_2) = \omega(\pi_1) \omega(\pi_2), \quad \text{for } \pi_1, \pi_2 \in \EF.
    \end{equation}
\end{theorem}
\begin{proof}
    We use the definition of $\omega$ and $\Delta_\sigma$ to have
    \begin{align*}
        \omega(\pi_1 \cdot \pi_2) &= \langle (A \circ \delta_\sigma) (a), \pi_1 \cdot \pi_2 \rangle_\sigma = \langle (\Delta_\sigma \circ A \circ \delta_\sigma) (a), \pi_1 \otimes \pi_2 \rangle_\sigma \\
        \shortintertext{let us use identity $(2)$ from Proposition \ref{prop:Alg_LD}}
        &= \langle ((A \otimes A) \circ \Delta_\sigma \circ \delta_\sigma) (a), \pi_1 \otimes \pi_2 \rangle_\sigma \\
        \shortintertext{we use the explicit formula for $\Delta_\sigma$ to obtain}
        &= \langle (A \circ \delta_\sigma) (a), \pi_1 \rangle_\sigma \langle (A \circ \delta_\sigma) (a), \pi_2 \rangle_\sigma = \omega(\pi_1) \omega(\pi_2).
    \end{align*}
    and this finishes the proof.
\end{proof}

\section*{Acknowledgements}

The author would like to thank Dominique Manchon, Gilles Vilmart, and the anonymous reviewers for useful suggestions and comments. This work was partially supported by the Swiss National Science Foundation, projects No 200020\_214819, No. 200020\_184614, and No. 200020\_192129.

\newpage 

\appendix

\section{All grafted and exotic trees up to size 3}

\begin{table}[!htbp]
    \centering
    \begin{tabular}{||c | c | c | c | c || c | c | c | c | c ||} 
         \hline
         $|\tau|$ & $(\tau, \alpha_g)$ & $\sigma(\tau, \alpha_g)$ & $(\tau, \alpha_e)$ & $\sigma(\tau, \alpha_e)$ & $|\tau|$ & $(\tau, \alpha_g)$ & $\sigma(\tau, \alpha_g)$ & $(\tau, \alpha_e)$ & $\sigma(\tau, \alpha_e)$ \\[0.5ex] 
         \hline\hline
         $0.5$ & $\forest{x}$            & $1$  &                         &     & $3$  & $\forest{b[b[b]]}$      & $1$  & $\forest{b[b[b]]}$      & $1$ \\
         \cline{1-5}
         $1$   & $\forest{b}$            & $1$  & \forest{b}              & $1$ &      & $\forest{b[b,b]}$       & $2$  & $\forest{b[b,b]}$       & $2$ \\
         \cline{1-5}
         $1.5$ & $\forest{b[x]}$         & $1$  &                         &     &      & $\forest{b[b,x,x]}$     & $2$  & $\forest{b[b,1,1]}$     & $2$ \\
               & $\forest{(b),x}$        & $1$  &                         &     &      & $\forest{b[b[x],x]}$    & $1$  & $\forest{b[b[1],1]}$    & $1$ \\
         \cline{1-5}
         $2$   & $\forest{b[b]}$         & $1$  & $\forest{b[b]}$         & $1$ &      & $\forest{b[b[x,x]]}$    & $2$  & $\forest{b[b[1,1]]}$    & $2$ \\
               & $\forest{b[x,x]}$       & $2$  & $\forest{b[1,1]}$       & $2$ &      & $\forest{b[x,x,x,x]}$   & $24$ & $\forest{b[1,1,2,2]}$   & $8$ \\
               & $\forest{(b),b}$        & $1$  & $\forest{(b),b}$        & $1$ &      & $\forest{(b,b),b}$      & $2$  & $\forest{(b,b),b}$      & $2$ \\
               & $\forest{(b[x]),x}$     & $1$  & $\forest{(b[1]),1}$     & $1$ &      & $\forest{(b[b]),b}$     & $1$  & $\forest{(b[b]),b}$     & $1$ \\
         \cline{1-5}
         $2.5$   & $\forest{b[b,x]}$     & $1$  &                         &     &      & $\forest{(b),b[b]}$     & $1$  & $\forest{(b),b[b]}$     & $1$ \\
                 & $\forest{b[b[x]]}$    & $1$  &                         &     &      & $\forest{(b[x,x]),b}$   & $2$  & $\forest{(b[1,1]),b}$   & $2$ \\
                 & $\forest{b[x,x,x]}$   & $6$ &                          &     &      & $\forest{(b[x]),b[x]}$  & $1$  & $\forest{(b[1]),b[1]}$  & $1$ \\
                 & $\forest{(b[x]),b}$   & $1$  &                         &     &      & $\forest{(b),b[x,x]}$   & $2$  & $\forest{(b),b[1,1]}$   & $2$ \\
                 & $\forest{(b),b[x]}$   & $1$  &                         &     &      & $\forest{(b,b[x]),x}$   & $1$  & $\forest{(b,b[1]),1}$   & $1$ \\
                 & $\forest{(b[x,x]),x}$ & $2$  &                         &     &      & $\forest{(b[b,x]),x}$   & $1$  & $\forest{(b[b,1]),1}$   & $1$ \\
                 & $\forest{(b,b),x}$    & $2$  &                         &     &      & $\forest{(b[b[x]]),x}$  & $1$  & $\forest{(b[b[1]]),1}$  & $1$ \\
                 & $\forest{(b[b]),x}$   & $1$  &                         &     &      & $\forest{(b[x,x,x]),x}$ & $6$  & $\forest{(b[1,1,2]),2}$ & $2$ \\
         \hline
    \end{tabular}
    \caption{All grafted and exotic trees up to size $|\tau| = 3$ with their symmetries $\sigma$ as described in Section \ref{sec:tree_formalism}. }
    \label{table:all_grafted_exotic_trees}
\end{table}

\newpage
\section{Order conditions for the invariant measure up to order 3}
\begin{table}[!htbp]
{\renewcommand{\arraystretch}{2}
\begin{tabular}{ | C{1.5cm} |C{2cm}|C{11.5cm}| }
\hline
    & $ \pi $ & $ \omega(\pi) $ \\
\hline
\hline
    order 1 
    & $ \forest{b} $ & $ \sum_{i=1}^s b_i - 1 $ \\
\hline
\hline
    order 2
    & $ \forest{b[0,0]} $ & $ \sum b_i d_i^2 - \frac{1}{2} + \sum b_i - 2 \sum b_i d_i $ \\
\cline{2-3}
    & $ \forest{b[b]} $ & $ \sum b_i a_{ij} - \frac{1}{2} + \sum b_i - 2 \sum b_i d_i $ \\
\cline{2-3}
    $\ast$ & $ \forest{b,b} $ & $ \sum b_i b_j + 1 - 2 \sum b_i $ \\
\hline
\hline
    order 3 
    & $ \forest{b[b,b]} $ & $ - 4 \sum b_i d_i a_{ij} + \sum b_i + \sum b_i a_{ij} a_{ik} - 4 \sum b_i d_i - \frac{1}{3} + 4 \sum b_i d_i^2 + 2 \sum b_i a_{ij} $ \\
\cline{2-3}
    & $ \forest{b[b[b]]} $ & $ 2 \sum b_i d_i b_j + \frac{3}{2} \sum b_i - \sum b_i b_j - \sum b_i d_i b_j d_j + \sum b_i a_{ij} a_{jk} - 2 \sum b_i d_i - \frac{1}{2} + \sum b_i a_{ij} - 2 \sum b_i a_{ij} d_j $ \\
\cline{2-3}
    & $ \forest{b[0,0,b]} $ & $ - 2 \sum b_i d_i a_{ij} + \sum b_i + \sum b_i d_i^2 a_{ij} - 4 \sum b_i d_i - \frac{1}{3} + 5 \sum b_i d_i^2 - 2 \sum b_i d_i^3 + \sum b_i a_{ij} $ \\
\cline{2-3}
    & $ \forest{b[b[0],0]} $ & $ \sum b_i d_i b_j - \sum b_i d_i a_{ij} + \sum b_i - \frac{1}{2} \sum b_i b_j - \frac{1}{2} \sum b_i d_i b_j d_j - 2 \sum b_i d_i - \frac{1}{3} + \sum b_i d_i^2 + \sum b_i a_{ij} - \sum b_i a_{ij} d_j + \sum b_i d_i a_{ij} d_j $ \\
\cline{2-3}
    & $ \forest{b[b[0,0]]} $ & $ 2 \sum b_i d_i b_j + \frac{3}{2} \sum b_i - \sum b_i b_j - \sum b_i d_i b_j d_j - 2 \sum b_i d_i - \frac{1}{2} + \sum b_i a_{ij} - 2 \sum b_i a_{ij} d_j + \sum b_i a_{ij} d_j^2 $ \\
\cline{2-3}
    & $ \forest{b[0,0,1,1]} $ & $ \sum b_i - 4 \sum b_i d_i + 6 \sum b_i d_i^2 - \frac{1}{3} - 4 \sum b_i d_i^3 + \sum b_i d_i^4 $ \\
\cline{2-3}
    $*$ & $ \forest{b,b,b} $ & $ \sum b_i b_j b_k + 3 \sum b_i - 3 \sum b_i b_j - 1 $ \\
\cline{2-3}
    $*$ & $ \forest{b[b],b}  $ & $ \sum b_i a_{ij} b_j - 2 \sum b_i d_i b_j - \frac{3}{2} \sum b_i + \sum b_i b_j + 2 \sum b_i d_i - \sum b_i a_{ij} + \frac{1}{2} $ \\
\cline{2-3}
    $*$ & $ \forest{b[0,0],b} $ & $ \sum b_i d_i^2 b_j - 2 \sum b_i d_i b_j - \frac{3}{2} \sum b_i + \sum b_i b_j + 2 \sum b_i d_i - \sum b_i d_i^2 + \frac{1}{2} $ \\
\hline
\end{tabular}
    }
    \caption{Order conditions for stochastic Runge-Kutta methods of order $3$ generated by the Algorithm 1. The rows marked by $\ast$ correspond to the order conditions which are automatically satisfied using the Theorem \ref{thm:Coef_mult}.}
	\label{table:append_ord_cond}
\end{table}

\newpage

\bibliography{diss}

\end{document}